\documentclass[11pt]{amsart}
\usepackage{graphicx}
\usepackage[utf8]{inputenc}
\usepackage[english]{babel}
\usepackage{eucal}
\usepackage{enumerate}
\usepackage[a4paper,
            left=3cm,
            right=3cm,
            top=3cm,
            bottom=3cm]{geometry}
\usepackage{amsthm}
\usepackage{amsfonts}
\usepackage{dsfont}
\usepackage{color}
\usepackage{todonotes}
\usepackage{amssymb}
\usepackage{amsmath}
\usepackage{graphicx}
\usepackage{mathtools}
\usepackage{cancel}

\usepackage{xcolor}
\usepackage{graphicx}
\usepackage[T1]{fontenc}
\usepackage{ stmaryrd }
\usepackage{enumitem} 
\usepackage{rotating}
\usepackage{tikz}
\usetikzlibrary{decorations.pathreplacing}
\usepackage{pict2e}
\usepackage{dsfont}
\usepackage{mathbbol}

\usepackage{hyperref}
\hypersetup{colorlinks=true,linkcolor=black,citecolor=black}    
\usepackage{ulem}
\newcommand{\N}{{\mathbb N}}

\newcommand{\IP}{{\mathbb P}}
\newcommand{\IE}{{\mathbb E}}

\newcommand\Z{\mathbb{Z}}

\newcommand{\cvlaw}{\stackrel{{ (d)}}{\longrightarrow}}

\newcommand*\cvLdeux{\overset{L^2}{\longrightarrow}}

\newcommand{\indep}{\perp \!\!\! \perp}

\def\E{{\mathbb E}}

\definecolor{atomictangerine}{rgb}{1.0, 0.6, 0.4}

\newtheorem{theorem}{Theorem}[section]
\newtheorem{lemma}[theorem]{Lemma}

\newtheorem*{theorem*}{Theorem}
\newtheorem*{lemma*}{Lemma}

\newtheorem*{proposition*}{Proposition}

\newtheorem{remark}[theorem]{Remark}

\theoremstyle{definition}

\newtheorem*{remark*}{Remark}
\newtheorem*{definition*}{Definition}

\theoremstyle{remark}

\begin{document}

\subjclass[2020]{Primary: 82D60; Secondary: 82B44, 60F05.}
\keywords{Directed Polymer in Random Environment, Critical space-correlation, Bessel function, Central Limit Theorem, Disordered Systems.}

  \title[CLT for 2D directed polymers with critical spatial correlation]{A central limit theorem for two-dimensional directed polymers with critical spatial correlation}
\author{Cl\'{e}ment Cosco, Francesca Cottini and Anna Donadini}
\address{Cl\'{e}ment Cosco, Ceremade, Université Paris Dauphine, Place du Maréchal de Lattre de Tassigny, 75775 Paris Cedex 16, France}
\email{cosco@ceremade.dauphine.fr}
\address{Francesca Cottini, LPSM, Sorbonne Université, 4 place Jussieu, 75005 Paris, France}
\email{francesca.cottini@sorbonne-universite.fr}
\address{Anna Donadini, Dipartimento di Matematica e Applicazioni, Università degli Studi di Milano-Bicocca, via Cozzi 55, 20125 Milano, Italy}
\email{a.donadini@campus.unimib.it}

\begin{abstract}
On the 1+2 dimensional lattice, we consider a directed polymer in a random Gaussian environment that is independent in time and correlated in space. The spatial correlation is supposed to decay as $(\log |x|)^a /|x|^{2}$, $a>-1$, where the square in the polynomial is known to be critical (\emph{Lacoin, Ann.\@ Prob.\@ (2011)}).  
    We introduce an intermediate regime of temperature $\beta_N \propto \hat \beta/(\log N)^{\frac{a+2}{2}}$, under which the log-partition function $\log W_N^{\beta_N}$ converges in distribution towards a Gaussian random variable if $\hat \beta\in (0,\hat \beta_c)$, whereas $W_N^{\beta_N}$ vanishes for $\hat \beta\geq \hat \beta_c$.
    The variance of the limiting Gaussian distribution, which
    is given by an inverse Bessel function, is determined by an induction scheme whose multi-scale dependence  reflects the critical nature of the correlation. 
    The Gaussianity of the limit follows from a decoupling argument of \emph{Cosco, Donadini 
    (2024+)}.
\end{abstract}
\maketitle

\section{Introduction}
Directed polymers in a random environment describe the behavior of a long, directed chain of monomers in the presence of random impurities. The model was originally introduced in \cite{huse} and received its first mathematical treatment in \cite{imbrie}. 

In its discrete setup, the polymer trajectory is a nearest-neighbor path $(S_n)_{n\in \mathbb N}$ on the $d$-dimensional lattice, while impurities (also called the environment) are given by random variables  $\omega(n,x)$ on $\mathbb{N}\times \mathbb{Z}^d$,
with $(\IP,\IE)$ denoting the associated probability measure and the expectation sign. 

Let $(P_x,E_x)$ be a probability measure and expectation under which $(S_n)_{n\in\mathbb{N}}$ is the simple random walk started at $x\in \mathbb Z^d$. (We write $(P,E)$ when $x=0$.) Then, the \emph{polymer measure} is defined in the Gibbsian sense by
\begin{equation*}
    dP_{N,\beta}^{\omega}(S)=\frac{e^{\beta\sum_{n=1}^N \omega(n,S_n)}}{Z_N^\beta}dP(S),
\end{equation*}
where $\beta>0$ is the inverse temperature, or disorder strength, and $Z_N^\beta$ is the normalizing constant called \emph{partition function}.

The classical setting, where $\omega(n,x)$ are independent in both time and space and admit exponential moments, has been extensively studied. We refer the reader to \cite{comets_book,Z24} for general introductions and to \cite{JL24} for significant recent developments.

In this article, we consider dimension $d=2$ with \emph{critical spatial correlation}, which corresponds to taking weights $\omega(n,x)$ that are jointly Gaussian, with zero mean and covariance
\begin{equation} \label{eq:correlation_def}
\IE \big[\omega(n,y)\omega(m,z)\big]=\mathds{1}_{n=m} \, h(y-z), \quad  \forall \, n,m\in \mathbb N, \ x,y\in \mathbb Z^2,
\end{equation}
where $h:\mathbb R^2\to \mathbb R_+$ is a \emph{bounded} function such that $h(0)=1$ and
\begin{equation}\label{eq:correlation_function}
    h(x)\sim \frac{(\log |x|)^a}{|x|^2} \text{ as } |x|\to\infty, \quad a>-1.
\end{equation}
We further assume that $h=h_0\star h_0$ for some non-negative function $h_0$ on $\mathbb Z^2$, where $\star$ denotes the convolution on $\mathbb Z^2$.

\begin{remark}
The polynomial decay of correlation $|x|^{-2}$ in \eqref{eq:correlation_function} is known to be critical in every dimension $d\geq 2$. Indeed,
Lacoin \cite{Lacoin-cor} showed that for spatial correlation $|x|^{-\alpha}$
with $\alpha > 2$, the model behaves similarly to the case of independent variables. In fact, the same conclusion holds as soon as $h\in L^1$, which includes the case where $a<-1$ in \eqref{eq:correlation_function}. In contrast, when $\alpha \in (0,2)$, the presence of correlation strongly impacts the model's properties and  superdiffusivity of the polymer path holds for any $\beta>0$. 
\end{remark}

In the above setting, the \emph{normalized partition function}
\begin{equation} \label{eq:defW}
W_N^\beta=W_N:=\frac{Z_N^\beta}{\IE[Z_N^\beta]} = E\left[e^{ \sum_{n=1}^N \{\beta \omega(n,S_n)-\frac{\beta^2}{2}\}}\right], \quad \beta >0\,,
\end{equation}
 can be shown to vanish as $N\to\infty$ for all fixed $\beta>0$. This phenomenon can be interpreted as a feature of \emph{strong disorder}, which in general implies localization properties for all $\beta>0$, see \cite{comets_book} for more details.

We consider an \emph{intermediate disorder} regime (see also \cite{AKQprl10,caravenna1} for $d=1,2$ with space-time independent noise and \cite{Rang20,ChGa23} for $d=1$ with spatially correlated environment), where $\beta=\beta_N$ vanishes as
\[
\beta_N \propto \frac{\hat \beta}{ (\log N)^{\frac{a+2}{2}}}, \quad \hat \beta > 0\,, \quad N\to\infty\,.
\] 
We show that in this regime, a phase transition occurs at some (explicit) parameter $\hat \beta_c > 0$, such that in the high temperature region $\hat \beta \in (0,\hat \beta_c)$, the log partition $\log W_N^{\beta_N}$
satisfies a central limit theorem,   whereas $W_N^{\beta_N}$ vanishes when $\hat \beta \in [\hat\beta_c,\infty)$.

This type of phase transition is reminiscent of dimension $d=2$ with space-time independent noise \cite{caravenna1,caravenna3,CoscoDonadini,DunlapGu}, where the same dichotomy occurs at the scale $\beta_N \propto \hat \beta (\log N)^{-1/2}$. However, our setup differs in the way that special constants (depending on Bessel functions) appear in the limiting variance. This results from a specific multi-scale structure 
relative to the criticality of the noise correlation.
As a matter of fact, the difficulty of the study relies on estimating the second moment $\IE[(W_N^{\beta_N})^2]$.
Indeed, one striking difference compared to the independent case, is that after Taylor expansion of the partition function, 
the second moment does not reduce to a geometric series whose ratio is given by the first order term of the expansion.
It is rather obtained by an induction scheme, where each exponential space window $[N^{(k-1)/2M},N^{k/2M}]$, for $ k\leq M$, contributes by a different factor.  

As a byproduct of this moment computation, we obtain a generalization of a classical theorem of Erd\"os and Taylor \cite{ErdosTaylor}, see Theorem \ref{th:ErdosTaylor}.

Once second moment computations are in hand, the proof of Gaussianity follows the strategy of \cite{CoscoDonadini}, arguing that $\log W_N^{\beta_N} \approx \sum_{k=1}^M \log Z_{k,M}^{\beta_N}$ where $Z_{k,M}^{\beta_N}$ are (independent) partition functions restricted to the time window $\big[N^{(k-1)/M},N^{k/M}\big]$, then concluding by the usual central limit theorem. This idea adapts nicely to our context thanks to the independence in time. 
\medskip

In the recent years, there has been growing interest in the related model of the stochastic heat equation (SHE) with correlated noise.
For the SHE with \emph{critical} correlation ($\alpha = 2$) in dimension $d\geq 3$, Mueller and Tribe  \cite{MuTri04} established the existence of a measure-valued solution to the equation,
inside some high temperature regime. More recently, limit theorems (of spatial averages) of the latter solution were investigated by Kuzgun and Tao \cite{KuTao24}. We also mention the recent work by Dunlap, Hairer and Li (\cite{DHL25}) 
on the non-linear SHE in $d\geq 3$ under critical correlation. 
Regarding subcritical correlation ($\alpha <2$), singular SPDEs can be constructed (see \cite{HuNuVi20,Dalang99}) and a central limit theorem for spatial averages has been obtained in \cite{HuNuVi20}. In dimension $d=1$, the SHE with correlated noise has been proven to arise as the scaling limit of discrete polymers (intermediate regime), see \cite{ChGa23,Rang20} for correlation in space and \cite{RaJiWa24} for results on long-range correlation in time. In the supercritical case ($\alpha > 2$), Gerolla, Hairer and Li have derived Gaussian fluctuations for the SHE and the KPZ (Kardar-Parisi-Zhang) equation in \cite{GeHaLi25,GeHaLi24}. We mention the work \cite{Kotitsas} that considers local correlations in time for the $2d$ SHE. 

\medskip

Finally, we refer to Section \ref{sec:openQuestions} for a number of related problems and open questions, including the behavior at the critical point $\hat \beta_c$, that we expect should yield an analogue of the Stochastic Heat Flow introduced in \cite{CaSuZyCrit21} (see also \cite{Nak25,TsaiMoments24}) for spatially colored noise.

\subsection{Main results}
Let $a>-1$. Set $\alpha := \frac{a+1}{a+2}-1>-1$ and  $\tilde J_\alpha(z) = \Gamma(\alpha+1) (z/2)^{-\alpha} J_{\alpha}(z)$, where
\begin{equation} \label{eq:defJalpha}
    J_\alpha(z) := \sum_{n=0}^{\infty}  \frac{(-1)^n}{\Gamma\big(n+\alpha + 1\big) \, n!}\,\Big( \, \frac{z}{2}\,\Big)^{2n + \alpha}\,
\end{equation}
is the Bessel function of the first kind.
Let
\begin{equation} \label{eq:defza}
z_a := \inf \{ z > 0 : \tilde J_{\alpha}(z)=0\}>0
\end{equation}
denote the first zero of $\tilde J_{\alpha}$ and define for all $\hat \beta < z_a$,
\begin{equation}\label{eq:lambda}
    \lambda^2 := \lambda^2(\hat \beta,a) = \log {\tilde J_\alpha(\hat \beta)}^{-1}.
\end{equation}
We consider the weak disorder regime $\beta_N\to 0$ where
\begin{equation} \label{eq:beta_N}
    \beta_N :=\frac{{\mathfrak C_a}\, \hat \beta}{(\log N)^{(a+2)/2}}, \qquad \text{with } \quad {\mathfrak C_a :=2^{\frac{a-1}{2}}\,(a+2)}.
\end{equation}

Our main result is the following. Let $a>-1$ and $W_N^\beta$ defined in \eqref{eq:defW} with weights satisfying \eqref{eq:correlation_def}. (Denote by $\mathcal N(\mu,\sigma^2)$ the  normal distribution of mean $\mu$ and variance $\sigma^2$.)
\begin{theorem}\label{thm:clt}
 For $\lambda^2$ as in \eqref{eq:lambda} and for all  $\hat \beta \in [0,z_a)$, 
\begin{equation}\label{eq:clt}
     \quad \log W_N^{\beta_N}\cvlaw \mathcal N\left(-\frac{\lambda^2}{2},\lambda^2\right),\quad N\to\infty\,,
\end{equation} 
in distribution.
On the other hand, $W_N^{\beta_N}\to 0$ in probability for all $\hat \beta\geq z_a$.
\end{theorem}
\begin{remark} \label{rk:a=0}
When $a=0$, we have $z_a = \pi/2$ and $\tilde J_{\alpha}(\hat \beta) = \cos(\hat \beta)$.
We note that $\lambda^2$ diverges to $\infty$ as $\hat \beta\to z_a$.
\end{remark}
\smallskip 

Our next result is a generalization of the Erd\"os-Taylor theorem \cite{ErdosTaylor}. The latter states that the rescaled local time $\frac{\pi}{\log N}\sum_{n=1}^N \mathds{1}_{S_{2n}=0}$ converges in law to an exponential distribution $\mathcal E(1)$ with mean 1. 
In line with our
setting,
the following happens if one replaces the 
delta potential $\mathds{1}_{x=0}$
by the long-range potential $h(x)$ with (critical) decay as in  \eqref{eq:correlation_function}.
\begin{theorem}
\label{th:ErdosTaylor}
As $N\to\infty$, the following convergence in distribution holds
\begin{equation} \label{eq:CVErdosTaylor}
    \frac{(a+2)^2\, 2^{a}}{\left(\log N\right)^{a+2}}  \sum_{n=1}^N h(S_{2n}) \cvlaw \tau_{0,1}^{\alpha}\,,
\end{equation}
where $\tau_{0,1}^{\alpha}$ is the first hitting time of $1$ of a Bessel process of order $2\alpha+2$ started at $0$.
\end{theorem}
We refer to \cite{Kent78} for a discussion about hitting times of Bessel processes and their connection to Bessel functions.
Theorem \ref{th:ErdosTaylor} is a consequence of the convergence of the moment generating functions obtained in Theorem \ref{thm:main}, see Section \ref{subsec:ErdosTaylorLaplace} for details.

Although the result above seems rather standard, we could not find its statement in the literature. However, as pointed out to us by C.\@ Tardif and L.\@ Béthencourt, a similar convergence can be found in the continuum setting (replace the random walk by a Brownian motion and the sum by an integral), for a function $h(x) = \phi(\log |x|)/|x|^2$ with $\phi:\mathbb R\to \mathbb R$ bounded and integrable
 (where our setting  \eqref{eq:correlation_function} considers a non-integrable function $\phi(x)\sim|x|^{-a}$ as $|x|\to\infty$ ($a>-1$) but assumes $h$ to be bounded at the origin), see Revuz and Yor \cite[p.\@432, Theorem 4.2]{revuzYor}. Because of this discrepancy, the scaling order and the limit obtained there differ from ours, but the same line of proof should adapt to deal with a continuous analogue of Theorem \ref{th:ErdosTaylor}. Essentially, the idea is to use the skew-product representation of the planar Brownian motion \cite[p.193, Theorem (2.11)]{revuzYor} and diffusive scaling to approximate  $\int_{0}^t h(B_s)\mathrm ds$ by $(\log t)^2 \int_{0}^{T_1} \phi((\log t)\beta_s)\mathrm ds$, where 
 %$\phi(\log |x|)=|x|^2h(x)$, 
 $\beta_s$ is a (dimension 1) Brownian motion and $T_1$ is its first hitting time of $1$. Then, the last integral can be approximated by $(\log t)^a \int_{0}^{T_1} \beta_s^a \mathbf{1}_{\beta_s > 0}\mathrm ds$, where by the use of a scale function and speed measure, the latter can be related to the first hitting time of a Bessel process. We stress that while this approach may lead to another proof of Theorem \ref{th:ErdosTaylor}, it is not clear that it could give a precise control of the moment generating function as in Theorem \ref{thm:main}, Equation \eqref{eq:secondMomentFlat}.

\subsection{Second moment estimates}\label{sec:second_moment}
The first steps leading to Theorem \ref{thm:clt} are results on the second moment of partition functions. In particular, we prove convergence of the second moment for the following objects: for $s\leq t \leq N$, for $x\in \mathbb{Z}^2$, let
\begin{equation}\label{eq:partition_function}
    W_{s,t}(x) =W_{s,t}^{\beta_N}(x):=E_x\left[e^{\sum_{n=s+1}^t\{\beta_N\omega(n,S_n)-\frac{\beta_N^2}{2}\}}\right], \quad W_t(x):=W_{0,t}(x).
\end{equation}
and consider for $0<\zeta <\xi\leq 1$, \begin{equation}\label{eq:partition_function_z}
      \Theta_{\zeta,\xi}(x)=\Theta_{\zeta,\xi}^{\beta_N}(x):=W_{\lfloor N^\zeta\rfloor,\lfloor N^\xi\rfloor}(x), \quad \Theta_\xi (x):= W_{\lfloor N^{\xi} \rfloor}(x)\,.
 \end{equation}
 For simplicity of notation, when $x=0$ we drop the dependence on the starting point and the integer part in quantities such as $N^\xi$.
\begin{theorem}\label{thm:main}
    Let $\hat \beta \in [0,z_a)$. Then,
    \begin{equation} \label{eq:secondMomentFlat}
\lim_{N\to\infty} \IE\left[W_{N}^2\right] = \lim_{N\to\infty} E\left[e^{\beta_N^2 \sum_{n=1}^N h(S_{2n})}\right] =  \tilde J_{\alpha}\big( \hat{\beta}\big)^{-1}\,.
\end{equation}

Moreover, for all $0\leq \zeta< \xi\leq 1$,
    \begin{equation} \label{eq:2ndMomentZetaXilimit}
        \lim_{N\to\infty} \IE\left[\Theta_{\zeta,\xi}^2\right] = \lim_{N\to\infty} E\left[e^{\beta_N^2 \sum_{n=N^{\zeta}+1}^{N^{\xi}} h(S_{2n})}\right] = \frac{\tilde J_{\alpha}\left( \hat{\beta}\zeta^{\frac{a+2}{2}}\right)}{\tilde J_{\alpha}\left( \hat{\beta}\,\xi^{\frac{a+2}{2}}\right)}\,.
    \end{equation}
Furthermore, for all $\delta >0$ and every $x_0\in [N^{\frac{\zeta - \delta}{2}},N^{\frac{\zeta+\delta}{2}}]\cap \mathbb{Z}_{even}^2$,
\begin{equation} \label{eq:2ndMomentSpaceVersion}
\IE\left[\Theta_\xi(x_0)\Theta_\xi(0)\right]= (1+o(1)) \frac{\tilde J_{\alpha}\Big( \hat{\beta}\zeta^{\frac{a+2}{2}}\Big)}{\tilde J_{\alpha}\Big( \hat{\beta}\xi^{\frac{a+2}{2}}\Big)},
\end{equation}
where $o(1)=\Delta_{\delta,N}(x_0)$ satisfies
\[
\limsup_{\delta\to 0} \limsup_{N\to\infty} \sup_{x_0\in [N^{\zeta - \delta},N^{\zeta+\delta}]\cap \mathbb{Z}_{even}^2} |\Delta_{\delta,N}(x_0)|=0.
\]

Finally, we have
\begin{equation} \label{eq:allsup2ndMoment}
    \sup_{N\geq 1} \sup_{x_0\in \mathbb{Z}^2}\IE[W_N(x_0)W_N(0)] < \infty.
\end{equation}
\end{theorem}

\subsection{A central limit theorem}
As mentioned in the introduction,
the proof of Theorem \ref{thm:clt} follows the approach of \cite{CoscoDonadini}, up to minor modifications. Below, we state a central limit theorem (Theorem \ref{thm:general}) that summarizes its main steps. In particular, the proof of Theorem \ref{thm:general} is obtained by combining Proposition 3.3, Lemma 4.1 and Lemma 4.2 of  \cite{CoscoDonadini}. 
\begin{theorem}[Log-CLT]\label{thm:general}
    Let $\{Z_N\}_{N \in \N}$ and $\{Z_{k,M,N}\}_{1 \le k \le M, \, M , \, N \in \N}$ be respectively a sequence of positive random variables and an array of random variables with mean one and finite variance, satisfying the following conditions:        \begin{enumerate}[label=(\roman*)]
            \item \label{item:indipendence} \emph{Independence:} for $M,N\in\mathbb N$, $\{Z_{k,M,N}\}_{ 1 \le k\leq M}$ is a collection of independent random variables.    
            \item \label{item:decoupling}\emph{Decoupling in $L^2$:} for $M\in\mathbb N$, $Z_N$ admits a multiplicative approximation in $L^2$, i.e.\ \begin{equation}
            \label{eq:decoupingCLT}
             Z_N - \prod_{k=1}^{M} Z_{k,M,N} \xrightarrow[N \to \infty]{L^2} 0\,.
            \end{equation}        
    \end{enumerate}
    Moreover, denoting the centered variables $U_{k,M,N}:=Z_{k,M,N}-1$, we assume:
    \begin{enumerate}[label=(\roman*), start=3]
        \item \label{item:secmomCLT} \emph{Second moment convergence:} there exists $\sigma_M^2>0$ and $\sigma^2>0$ such that
        \begin{equation}
            \label{eq:secmomCLT}
        \lim_{N \to \infty} \sum_{k=1}^M \mathbb{E} \big[ U_{k,M,N}^2 \big] = \sigma_M^2\,, \, \ \text{for }M \in \N \qquad \text{and} \qquad \lim_{M \to \infty} \sigma_M^2 = \sigma^2\,.
        \end{equation}
        \item \label{item:Lindeberg} \emph{Lyapunov condition:}
        \begin{equation} \label{eq:Lindeberg}\exists \, \varepsilon_0>0 : \quad \limsup_{M\to\infty} \,\limsup_{N\to \infty}\, \sum_{k=1}^M \mathbb{E}\big[ \big|U_{k,M,N}\big|^{2+\varepsilon_0} \big]=0\,.
        \end{equation}
    \end{enumerate}
Then, $\log Z_N \, \cvlaw \, \mathcal{N}\left(-\frac{\sigma^2}{2},\sigma^2\right)\,.$
\end{theorem}
\begin{remark}
    We emphasize that there is no need to assume that the $Z_{k,M,N}$ are positive. This is a consequence of the Lyapunov condition  \ref{item:Lindeberg}, which  yields some concentration of $Z_{k,M,N}$ around its mean. See e.g.\ equation (25) in \cite{CoscoDonadini} (and replace the supremum there by a sum).
\end{remark}

\begin{remark}
    The Lyapunov condition \ref{item:Lindeberg} can be relaxed to the usual Lindeberg condition:
    \begin{equation*}
        \forall \, \varepsilon_0>0 : \quad \limsup_{M\to\infty} \,\limsup_{N\to \infty}\, \sum_{k=1}^M \mathbb{E}\big[ \big(U_{k,M,N}\big)^{2}\,\mathds{1}_{\{|U_{k,M,N}|> \varepsilon_0\}}  \big]=0\,.
    \end{equation*}
\end{remark}

By Theorem \ref{thm:general}, the proof of Theorem \ref{thm:clt} reduces to verifying points \ref{item:indipendence}-\ref{item:Lindeberg}  with $Z_N:=W_N$ and a well-chosen $\{Z_{k,M,N}\}_{1\leq k\leq M}$. Regarding the latter quantity, we set
(let $t_0=0$)
\begin{equation} \label{eq:defZkMN}
Z_{k,M,N} \coloneqq E\left[\prod_{n=t_{k-1}+1}^{t_{k}} \left\{1+\beta_N\omega(n,S_n)\right\}\right], \quad t_k = \lceil N^{\frac{k}{M}}\rceil \quad 1\leq k\leq M\,.
\end{equation}

By construction, 
 $\{Z_{k,M,N}\}_{1\leq k\leq M}
$ clearly satisfies condition
\ref{item:indipendence}.
In Section \ref{sec:decoupling}, we show that condition \ref{item:decoupling} holds and Section \ref{sec:increments} deals with points \ref{item:secmomCLT} and \ref{item:Lindeberg}.

\begin{remark}
    \label{rem:Productformula}
    By analogy with \cite{CoscoDonadini}, a more natural choice for \( Z_{k,M,N} \) would be the quantity  $\Theta_{\frac{k-1}{M},\frac{k}{M}} = E[\exp(\sum_{n=t_{k-1}}^{t_k}\{\beta_N \omega(n,S_n)-\beta_N^2/2\}) ]$. In fact, $W_N\sim \prod_{k\leq M} \Theta_{\frac{k-1}{M},\frac{k}{M}}$ holds in $L^2$ and outlines the  asymptotic independence of the contributions of the different time windows $[t_{k-1},t_{k}]$. But the difficulty with this choice is to verify point \ref{item:Lindeberg}.
    Considering instead the product form \eqref{eq:defZkMN}
   allows to write $Z_{k,M,N}$ as a polynomial chaos
   in \emph{independent} (Gaussian) random variables,
   see Lemma \ref{lemma:fourth_moment}, in particular Equation \eqref{eq:X_mDecU_N}.
   Then, the Lyapunov estimate \ref{item:Lindeberg} derives from  \emph{hypercontractivity for polynomial chaos} \cite{caravenna2}, for which independence of the chaos variables is a crucial assumption.
\end{remark}

\subsection*{Structure of the paper.}
Section \ref{sec:2ndmoment} gathers proofs and estimates on second moment results presented in Section \ref{sec:second_moment} as well as the proof of the Erd\"os-Taylor Theorem \ref{th:ErdosTaylor}. In particular, the latter is proved in Section \ref{subsec:ErdosTaylorLaplace}, while the second moment convergence given in Theorem \ref{thm:main} is shown in Section \ref{subsec:proofOf2ndMoment}.

The rest of the paper is devoted to proving the log-normality result stated in Theorem \ref{thm:clt}, which follows from Theorem \ref{thm:general} once verified its assumptions. In particular, Section \ref{sec:decoupling} establishes an $L^2$-decoupling, while Section \ref{sec:increments} provides higher moment estimates. Together, these properties ensure the conditions required in Theorem \ref{thm:general}, leading to a log-normal limit structure. The proof of Theorem \ref{thm:clt} is concluded in Section \ref{sec:end_of_proof}. 

Section \ref{sec:recursive_formula} provides the detailed computations underlying the recursive structure that culminates in the limit \eqref{eq:2ndMomentZetaXilimit}.

Finally, Section \ref{sec:openQuestions} is devoted to concluding comments and related open problems that may be of interest for future research.

Appendices \ref{app:techtools} and \ref{sec:approximation} collect proofs of technical lemmas exploited throughout the paper and omitted in the main text.

\subsection*{Acknowledgments} The authors are deeply grateful to Quentin Berger and Gaspard Gomez for the insightful discussions and their precious advice regarding the proof of Theorem \ref{lem:recursiveLemmaFormula} and to Alberto Chiarini for suggesting that the approach of \cite{CoscoDonadini} may extend to correlated environment.
They also thank Camille Tardif and Loïc Béthencourt for
the discussion below Theorem \ref{th:ErdosTaylor}, Francesco Caravenna, Nikos Zygouras, Rongfeng Sun and Ofer Zeitouni for their helpful comments and advice, Jhih-Huang Li and Shuta Nakajima for discussions that motivated this work. 

Most of the research for the present paper was carried out while F.C. was supported by the Luxembourg National Research Fund (AFR/22/17170047/Bilateral-GRAALS).
F.C. also acknowledges support from the European Union’s Horizon 2020 research and innovation programme under the Marie Skłodowska-Curie grant agreement No 101034255. F.C.\ and A.D.\ acknowledge the support of INdAM/GNAMPA. 

\section{Second moment computations} 
\label{sec:2ndmoment}

Theorem \ref{thm:general} morally reduces the proof of Theorem \ref{thm:clt} to controlling second and $2+\varepsilon$ moments of $W_{s,t}$. 
The following lemma yields a close expression of the second moment, in terms of an expectation with respect to the simple random walk.
\begin{lemma}[Covariance formula]\label{lem:second_moment}
Let $0\leq s<t \leq N$ and $x,y\in\mathbb{Z}^2$. Then
\begin{equation}\label{eq:second_moment}
    \IE[W_{s,t}(x)W_{s,t}(y)]=E_{x-y}\left[e^{\beta_N^2\sum_{n=s+1}^{t} h(S_{2n})}\right].
\end{equation}
Furthermore, with $\bar \sigma_N^2 := e^{\beta_N^2 } \beta_N^2$,
    \begin{equation} \label{eq:secondMomentFormulast}
    E_{x}\left[\prod_{n=s+1}^{t} \{1+\beta_N^2 h(S_{2n})\}\right] \leq \IE[W_{s,t}(x)W_{s,t}(0)] \leq E_{x}\left[\prod_{n=s+1}^{t} \{1+\bar \sigma_N^2  h(S_{2n})\}\right].
    \end{equation}
\end{lemma}
\begin{remark}
    In the i.i.d.\ setting where $\omega(n,x)$ are space-time independent Gaussians, the identity \eqref{eq:second_moment} holds by replacing $h(x)$ by the Dirac function  $\mathds{1}_{x=0}$. 
\end{remark}

\begin{proof} By Fubini's theorem,
    \begin{align*}
        \IE[W_{s,t}(x)W_{s,t}(y)]
        &=\IE\left[E_{x,y}^{\otimes 2}\left[e^{\sum_{n=s+1}^{t}\{\beta_N\omega(n,S_n^1)+\beta_N\omega(n,S_n^2)-\beta_N^2\}}\right]\right]\\
        &=E_{x-y}\left[e^{\beta_N^2\sum_{n=s+1}^{t} h(S_{2n})}\right],
    \end{align*}
    where the last equality holds since $(\omega(n,x))_{(n,x)\in\mathbb{N}\times \mathbb{Z}^2}$ is a collection of Gaussian variables.
    By Taylor approximation, for all $x\in \mathbb Z^2$,
$1+\beta_N^2 h(x) \leq e^{\beta_N^2 h(x)} \leq 1+e^{\beta_N^2\Vert h \Vert_{\infty}}\beta_N^2 h(x)$, which yields the second part of the lemma (observe that necessarily $\Vert h \Vert_\infty = 1$ since $h(x) = \IE[\omega(n,x) \omega(n,0)]$, by the Cauchy-Schwarz inequality and $h(0)=1$).
\end{proof}

In the second moment estimates of Sections
    \ref{subsec:upperbound} and \ref{subsec:lowerBound}, we will consider a general term of the form
$E_{x}\left[\prod_{n=s+1}^{t} \{1+\sigma_N^2  h(S_{2n})\}\right]$, where $\sigma_N^2 \sim \beta_N^2$ as $N\to\infty$. In particular, the term $\sigma_N^2$ will be identified with $\bar \sigma_N^2$ in the upper bound and directly with $\beta_N^2$ in the lower bound.
This is convenient to apply the following formula, whose proof is self-contained.

\begin{lemma}[Second moment expansion]
We have
    \begin{align} 
&E_{x}\left[\prod_{n=s+1}^{t} \{1+\sigma_N^2 h(S_{2n})\}\right]=1+\sum_{k=1}^\infty \sigma_N^{2k} \sum_{s+1\leq n_1<\dots<n_k\leq t} E_{x}\left[\prod_{i=1}^k h(S_{2n_i})\right] \nonumber\\
    &=\sum_{k=0}^\infty \sigma_N^{2k} \sum_{s+1\leq n_1<\dots<n_k\leq t}\,\sum_{x_1,\dots,x_k\in\mathbb{Z}^2} \prod_{i=1}^k  h(x_i)p_{2(n_i-n_{i-1})}(x_i-x_{i-1})\label{eq:second_moment_chaos}\,,
    \end{align}
    where $n_0=0$, $x_0=x$ and $\sigma_N^2>0$.
\end{lemma}

\subsection{Technical tools for the second moment}\label{sec:technicaltools}
In the following, we present two lemmas that will be repeatedly used to derive upper and lower bounds on the second moments. The proofs can be found in Appendix \ref{app:techtools}.

Fix $x \in \mathbb{Z}_{\text{even}}^2\coloneqq \big\lbrace x=(x^{(1)}, x^{(2)})\in\mathbb{Z}^2 \, : \, x^{(1)}+x^{(2)}\in 2 \Z \big\rbrace$. 
In the following, we give estimates on the asymptotic behaviour of 
\begin{equation}\label{eq:RN}
    R_{s,t}(x):=\sum_{n=s+1}^t p_{2n}(x), \qquad R_t(x) := R_{0,t}(x)\,, \qquad s,t\in\mathbb{N}, s<t\,. 
\end{equation}
(If $s$ or $t$ are real, we take the lower and upper integer parts of $s$ and $t$ in \eqref{eq:RN}).

Let $M\in \mathbb{Z}_{+}$ and 
let $B(r)$ be the $\ell^2$ norm open ball of $\mathbb R^2$ of radius $r>0$. We set
\begin{equation} \label{eq:Al}
A_0 = B(N^{\frac{1}{2M}}),\quad  A_l =   (B(N^{(l+1)/(2M)})\setminus B(N^{l/(2M)})) \cap \mathbb Z^2_{even}\,,\, l \in \mathbb{Z}_{+}\,.
\end{equation}

\begin{lemma}[Order of $R_N$]\label{lem:order_of_RNx} Let $0\leq \nu<\xi\leq 1$ and $M\in \mathbb{Z}_{+}$. Then, for all $N$ large enough (depending on $M,\nu,\xi$), 
for all integers $0\leq l<l'\leq \lceil \xi M \rceil$, $|l-l'|>1$, every $x\in A_l$ and $y\in A_{l'}$,
\begin{equation}\label{eq:order_of_RNx}
        R_{N^{\nu},N^{\xi}}(x-y)\leq \begin{cases}
            (\xi-\frac{l'}{M}+ \frac{3}{M})\frac{\log N}{\pi} \quad \text{if } \nu<\frac{l'}{M}<\xi,\\
            (\xi-\nu + \frac{3}{M})\frac{\log N}{\pi} \quad \text{if } \frac{l'}{M}<\nu,
        \end{cases}
    \end{equation}
    and
    \begin{equation}\label{eq:order_of_RNx_bis}
        R_{N^{\nu},N^{\xi}}(x-y)\geq  \begin{cases}
            (\xi-\frac{l'}{M}-\frac{3}{M})\frac{\log N}{\pi} \quad \text{if } \nu<\frac{l'}{M}<\xi,\\
            (\xi-\nu-\frac{3}{M})\frac{\log N}{\pi} \quad \text{if } \frac{l'}{M}<\nu.
        \end{cases}
    \end{equation}
Furthermore, for all $x,y\in\mathbb{Z}_{\text{even}}^2$ such that $|x|\leq N^{\frac{l}{2M}}$ and $|y|\geq N^{\frac{l+1}{2M}}$,
\begin{equation}\label{eq:order_RN_more_general}
    R_{N^\xi}(x-y)\leq 2 R_{N^\xi}(y)\leq 4 R_{N^\xi}(x),
\end{equation}
Finally, for some universal constant $c>0$, 
\begin{equation} \label{eq:ChernoffBound}
    \sup_{|y| \leq N^{\xi/2}}\sum_{|x|\geq N^{(\xi+1/M)/2}} R_N(x-y) \leq e^{-cN^{1/M}}.
\end{equation}
\end{lemma}
\medskip

The next result concerns the sum of $h$ over the exponentially growing windows $A_l$.
\begin{lemma}\label{lem:bound_h}
    Let $a>-1$, $M\in \mathbb{Z}_{+}$ and $\varepsilon>0$.  Then, there exist $N_0=N_0(M,\varepsilon)$ such that for all $N\geq N_0,$ and for all $l\in\{1,\dots,M\}$,
\begin{equation}{\label{eq:sumhAl}}
(1-\varepsilon)\pi   \left(\frac{\log N}{2M}\right)^{a+1}\left(l+\mathds{1}_{a\leq0}\right)^a\leq \sum_{x\in A_{l}} h(x) \leq (1+\varepsilon)\pi 
   \left(\frac{\log N}{2M}\right)^{a+1}{\left(l+\mathds{1}_{a>0}\right)^a}\,.
\end{equation} 
Moreover, $\sum_{x\in A_0} h(x) \leq C_a \left(\frac{\log N}{M}\right)^{a+1}$ where $C_a>0$.
\end{lemma}

\subsection{The upper bound}\label{subsec:upperbound}
In this section, we derive an upper bound for the quantity \begin{equation} \label{eq:defPhi}
\Phi_{t}(x_0) = \Phi_{t}^{\hat \beta}(x_0) := E_{x_0} \left[\prod_{n=1}^{\lfloor t\rfloor} \left\{1+\sigma_N^2 h(S_{2n})\right\}\right]\,, \quad t\leq N, x_0\in \Z^2,
\end{equation}
with $\sigma_N^2>0$ chosen such that 
$\sigma_N^2\sim \beta_N^2$ as $ N\to\infty$.

The following lemma provides key upper bounds that serve as building blocks in the proofs of Theorems \ref{thm:clt} and \ref{thm:main}. 

\begin{lemma}[Upper bound]\label{thm:upper_bound}
Let $0\leq \zeta <\xi\le 1$ and recall \eqref{eq:Al}. Then, for all $\hat \beta \in [0,z_a)$:
\begin{equation}
    \label{eq:upper_bound}
\limsup_{\delta\to 0}\, \limsup_{N\to\infty}\sup_{x_0\in [N^{\zeta - \delta},N^{\zeta+\delta}]\cap \Z^2_{even}} \Phi_{N^\xi}(x_0)\leq \frac{\tilde J_{\alpha}\Big( \hat{\beta}\zeta^{\frac{a+2}{2}}\Big)}{\tilde J_{\alpha}\Big( \hat{\beta}\xi^{\frac{a+2}{2}}\Big)}.
\end{equation}
In particular,
\begin{equation}\label{eq:finiteness2ndMoment}
   \sup_{N\geq 1} \sup_{x_0\in\mathbb{Z}^2} \,\Phi_{N}(x_0) <\infty\,.
\end{equation}

\end{lemma}
\begin{proof}
We first prove the upper bound \eqref{eq:upper_bound}. 
The statement in \eqref{eq:finiteness2ndMoment} will be proved at the end of the section.

Recall the definition of $A_l$ in \eqref{eq:Al}. Let $\varepsilon>0$ be arbitrary. By continuity of $\tilde J_\alpha$, the bound \eqref{eq:upper_bound} follows if we can show that for $M>0$ large enough, there exists $N_0=N_0(M,\varepsilon)$ such that
\begin{equation} \label{eq:upper_boundM}
\forall N\geq N_0, \ \forall \, l_0 \leq \lceil \xi M \rceil,\quad \sup_{x_0\in A_{l_0}} \Phi_{N^{\xi}}(x_0)\leq (1+\varepsilon)\frac{\tilde J_{\alpha}\Big( \hat{\beta}(\frac{l_0}{M})^{\frac{a+2}{2}}\Big)}{\tilde J_{\alpha}\Big( \hat{\beta}\xi^{\frac{a+2}{2}}\Big)}\,.
\end{equation}
Let $n_0=0$; then, by \eqref{eq:second_moment_chaos}, for all $x_0 \in \mathbb Z^2$, we get
\begin{align}\nonumber \Phi_{N^\xi}(x_0)   &= \sum_{k=0}^\infty \sigma_N^{2k} \sum_{1\leq n_1<\dots<n_k\leq \lfloor N^\xi\rfloor}\,\sum_{\substack{x_1,\dots,x_k\in\mathbb{Z}_{even}^2}}\prod_{i=1}^kp_{2(n_i-n_{i-1})}(x_i-x_{i-1})h(x_i)\\ \nonumber
    &\leq \sum_{k=0}^\infty \sigma_N^{2k}\sum_{\substack{1 \le u_1,\dots,u_k\le \lfloor N^\xi\rfloor}}\,\sum_{\substack{x_1,\dots,x_k\in\mathbb{Z}_{even}^2}}\prod_{i=1}^kp_{2u_i}(x_i-x_{i-1})h(x_i),
\end{align}
where the inequality follows by enlarging the sum over $u_i := n_{i}-n_{i-1} \in \{1,\ldots,N^\xi\}$ for $i= 1,\ldots, k$.
Therefore, by \eqref{eq:RN} we have 
\begin{equation}\label{eq:gath1}
    \Phi_{N^{\xi}}(x_0) \leq
    \sum_{k=0}^\infty \sigma_N^{2k}\sum_{\substack{x_1,\dots,x_k\in\mathbb{Z}_{even}^2}}\prod_{i=1}^kR_{N^\xi}(x_i-x_{i-1})h(x_i).
\end{equation}
Let $M\in\mathbb{Z}_{+}$ be a parameter that will be taken arbitrarily large and recall the definition of the annuli $(A_l)_{l\geq 0}$ in  \eqref{eq:Al}. Define for all $l_0\in \mathbb Z_+$:
\[
\Psi_{\xi , N}(l_0) :=  \sup_{x_0\in A_{l_0} }
    \sum_{k=0}^\infty \sigma_N^{2k} \sum_{ \substack{ \mathbf {\bf l} = (l_1,\dots,l_k) \in \mathbb Z_+^k}}\sum_{\substack{x_i\in A_{l_i}\\ \forall i=1,\dots, k}} \prod_{i=1}^kR_{N^\xi}(x_i-x_{i-1})h(x_i),
\]
which by \eqref{eq:gath1} satisfies
\begin{equation} \label{eq:PhiToPsi}
\forall l_0\in \mathbb Z_+,\quad \sup_{x_0\in A_{l_0}} \Phi_{N^{\xi}}(x_0) \leq \Psi_{\xi , N}(l_0).
\end{equation}
For ${\bf l} = (l_1,\ldots,l_k) \in \mathbb{Z}_{+}^k$, we introduce the following: for $j \in \{1,\ldots,k\}$ we say $l_j$ to be ``bad'' if
\begin{align}\label{eq:bad1}
&\text{either } \quad l_j \in \{0\}\cup \{\lceil \xi M \rceil -1,\lceil \xi M \rceil,\dots\}\,, \\
&\text{or }  \quad \quad \ l_{j-1}, l_j \in \{1,\ldots, \lceil \xi M \rceil -2\} \text{ with } |l_j - l_{j-1}| \le 1 \,,\label{eq:bad2}
\end{align}
where from now on we drop the notation of the integer part for simplicity.
We call $l_j$ ``good'' when it is not bad. Then, for any $k \in \N$, we introduce
\begin{align}
C_k &:=\Big\lbrace  {\bf l}  \in \mathbb{Z}_{+}^k : \  l_i \text{ is ``good''} \ \forall \, i \in \{1,\ldots,k\}  \Big\rbrace \label{eq:defCk}\\
&=\Big\lbrace \mathbf {\bf l}  \in \mathbb{Z}_{+}^k : \, l_1,\ldots,l_k \in \{ 1,\ldots,\xi M - 2 \} \text{ and } |l_i-l_{i-1}|> 1, \text{ for all } i=1,\ldots,k\Big\rbrace \,. \nonumber
\end{align}
We emphasize that $C_k$ depends on $l_0\in \mathbb{Z}_{+}$. 

Next, we decompose
\begin{equation}\label{eq:decomposition}
    \Psi_{\xi,N}(l_0)\leq \Psi^{lead}_{\xi,M,N}(l_0) + \Psi^{negl}_{\xi,M,N}\,,
\end{equation}
where
\begin{align}
    \Psi^{lead}_{\xi,M,N}(l_0) &:= \sup_{x_0\in A_{l_0}}\sum_{k=0}^\infty \sigma_N^{2k} \sum_{ \substack{ \mathbf {\bf l} \in C_k}}\sum_{\substack{x_i\in A_{l_i}\\ \forall i=1,\dots, k}} \prod_{i=1}^kR_{N^\xi}(x_i-x_{i-1})h(x_i)\,,\label{eq:PSILEADbd1} \\ \label{eq:psinegldef}
    \Psi^{negl}_{\xi,M,N} & := \sup_{l_0\in \mathbb{Z}_+} \sup_{x_0\in \mathbb{Z}_{+}}\sum_{k=0}^\infty \sigma_N^{2k} \sum_{ \substack{ \mathbf {\bf l} \not \in C_k}}\sum_{\substack{x_i\in A_{l_i}\\ \forall i=1,\dots, k}} \prod_{i=1}^kR_{N^\xi}(x_i-x_{i-1})h(x_i)\,.
\end{align}
We claim that
\begin{equation}\label{eq:psinegl}
    \limsup_{M \to \infty}\,\limsup_{\varepsilon \to 0}\,\limsup_{N\to\infty} \Psi^{negl}_{\xi,M,N} = 0\,.
\end{equation}
and that for $M$ large enough (depending on $\varepsilon$), there exists $N_0=N_0(M,\varepsilon)$ such that 
\begin{equation} \label{eq:EstimatePsilead}
   \forall N\geq N_0, \ \forall \, l_0\leq \xi M,\quad  \Psi^{lead}_{\xi,M,N}(l_0)\leq (1+\varepsilon)\frac{\tilde J_{\alpha}\Big( \hat{\beta}(\frac{l_0}{M})^{\frac{a+2}{2}}\Big)}{\tilde J_{\alpha}\Big( \hat{\beta}\xi^{\frac{a+2}{2}}\Big)}\,.
\end{equation}
Combining these two claims with \eqref{eq:decomposition} yields \eqref{eq:upper_boundM}. 
\end{proof}

\begin{proof}[Proof of \eqref{eq:EstimatePsilead}] 
As the spatial variables \( x_i \), $i=1,\ldots,k$, lie in diffusive balls with well separated exponents (i.e.\@ $l_i\neq l_{i-1}\leq M$), the overlaps \( R_{N^\xi}(x_i - x_{i-1}) \) for \( i = 1, \ldots, k \) can be estimated by using  Lemma~\ref{lem:order_of_RNx}. For $N\geq N_0$ with some $N_0=N_0(M)$, for all $l_0\leq \xi M$, we obtain
\begin{align}\label{eq:PSILEADbd2}
\Psi^{lead}_{\xi,M,N}(l_0)&\leq \sum_{k=0}^\infty \sigma_N^{2k}
\sum_{\substack{1\leq l_1, \dots, l_k\leq \xi M -2\\|l_i-l_{i-1}|>1}} \,\sum_{\substack{x_i\in A_{l_i}\\ i=1,\ldots,k}}
\prod_{i=1}^k\left\{\left(\xi-\frac{l_i}{M}\vee \frac{l_{i-1}}{M}+\frac{3}{M}\right)\frac{\log N}{\pi} \right\}h(x_i).
\end{align}
Appealing then to Lemma \ref{lem:bound_h},  we find that for $N$ large (depending on $\varepsilon,M$),
\begin{equation*} \label{eq:final_est_ub} 
    \Psi^{lead}_{\xi,M,N}(l_0) \leq \sum_{k=0}^\infty
\left((1+\varepsilon)\sigma_N^2\frac{(\log N)^{2+a}}{2^{a+1}}\right)^{k} I_{k,\zeta,\xi,M},
\end{equation*}
where
\begin{align}\label{eq:discrintegral}
I_{k,\zeta,\xi,M}  := \left(\frac{1}{M}\right)^k 
\sum_{\substack{1\leq l_1, \dots, l_k\leq \xi M -2\\|l_i-l_{i-1}|>1}} \,
    \prod_{i=1}^k \left(\xi-\frac{l_i}{M}\vee \frac{l_{i-1}}{M}+\frac{3}{M}\right)\left(\frac{l_i+\mathds{1}_{a>0}}{M}\right)^a \,.
\end{align}
Since $\sigma_N^2 \sim \beta_N^2 \sim \mathfrak C_a ^2 \, \hat \beta ^2 / (\log N)^{a+2}$ as $N \to \infty$,
so by \eqref{eq:beta_N}, we further get that, for $N$ large,
\begin{equation} \label{eq:final_est_ub2} 
    \Psi^{lead}_{\xi,M,N}(l_0) \leq \sum_{k=0}^\infty\left((1+\varepsilon){\frac{a+2}{2}} \hat \beta \right)^{2k}  I_{k,\zeta,\xi,M}\,.
\end{equation}

Next, we estimate $I_{k,\zeta,\xi,M}$ by comparison to an integral. This can be done using monotonicity properties thanks to the restrictions on the sets of summation of the $l_i$'s that we chose.
If $a>0$, we first perform a change of variables $l_i + 1 \to l_i$. In any case,
we find
\begin{equation*}
    \begin{split}
I_{k,\zeta,\xi,M}  &\le \left(\frac{1}{M}\right)^k 
\sum_{\substack{1\leq l_1, \dots, l_k\leq \xi M -1}} \,
    \prod_{i=1}^k \left(\xi +\frac{4}{M} -\frac{l_i}{M}\vee \frac{l_{i-1}}{M}\right)\left(\frac{l_i}{M}\right)^a \\
    &= \left(\frac{1}{M}\right)^k \int_{[1,\xi M]^k} \prod_{i=1}^k \left( \xi +\frac{5}{M} -\frac{\lfloor \ell_i\rfloor}{M}\vee \frac{\lfloor \ell_{i-1}\rfloor}{M}\right)\left(\frac{\lfloor t_i\rfloor}{M}\right)^a \text{d}\ell _1\ldots\text{d}\ell_k \\
    &\le \left(\frac{1}{M}\right)^k \int_{[0,\xi M]^k} \prod_{i=1}^k \left( \xi +\frac{5}{M} -\frac{ t_i}{M}\vee \frac{ \ell_{i-1}}{M}\right)\Big(\frac{ \ell_i}{M}\Big)^a \text{d}\ell_1\ldots\text{d}\ell_k\\
    &= \int_{[0,\xi]^k} \prod_{i=1}^k \left( \xi + \frac{5}{M} - \ell_i \vee t_{i-1} \right)t_i^a \,\text{d}t_1\ldots\text{d}t_k \,,
\end{split}
\end{equation*}
where we have set $\ell_0=l_0$ and $t_0=l_0/M$, and the last inequality follows from exploiting the monotonicity of the functions $(l,l') \mapsto \xi -\frac{l}{M}\vee \frac{l'}{M}$
and $l \mapsto \big(\frac{l}{M}\big)^a$ (if $a<0$, we also change of variables $\ell_i-1\to \ell_i)$. By positivity of the integrand,
\begin{equation*}
\begin{split}
    I_{k,\zeta,\xi,M}
    &\leq \int_{\big[0,\xi+\frac{5}{M}\big]^k} \prod_{i=1}^k \left( \xi + \frac{5}{M} - t_i \vee t_{i-1} \right)t_i^a \,\text{d}t_1\ldots\text{d}t_k \\
    & = \Big( \xi + \frac{5}{M}\Big)^{(2+a)k} \int_{[0,1]^k} (1 - s_i \vee s_{i-1} ) s_i^a \,\text{d}s_1\ldots\text{d}s_k\,,
    \end{split}
\end{equation*}
where the last equality follows from the change of variables $t_i=(\xi + \frac{5}{M})s_i$ for $i=1,\ldots,k$ and $s_0=\frac{t_0}{\xi + \frac{5}{M}}=\frac{l_0}{M(\xi + \frac{5}{M})}$.
Coming back to \eqref{eq:final_est_ub2}, we obtain that
\begin{equation} \label{eq:firstBoundPsiLead}
    \begin{split}
        \Psi^{lead}_{\xi,M,N}(l_0)& \leq \sum_{k=0}^\infty\left((1+\varepsilon){\frac{a+2}{2}} \hat \beta \Big( \xi + \frac{5}{M}\Big)^{\frac{a+2}{2}}\right)^{2k}  \int_{[0,1]^k} (1 - s_i \vee s_{i-1} ) s_i^a \,\text{d}s_1\ldots\text{d}s_k\\
        & = \sum_{k=0}^\infty z^{2k} f_k (s_0) = \frac{\tilde{J}_\alpha \left( {\frac{2}{a+2}} z (\frac{l_0}{M(\xi + \frac{5}{M})})^{\frac{a+2}{2}}\right)}{\tilde{J}_\alpha( {\frac{2}{a+2}} z)} \, ,
    \end{split}
\end{equation}
where we followed the notation of Section \ref{sec:recursive_formula} (recall \eqref{eq:recursiveformula}) and set $z:= (1+\varepsilon){\frac{a+2}{2}} \hat \beta \Big( \xi + \frac{5}{M}\Big)^{\frac{a+2}{2}}$, and where the last equality above follows from Theorem \ref{lem:recursiveLemmaFormula}, provided that $z < (a+2)z_a/2$, which is verified if we take $\varepsilon$ small enough and $M$ large enough (in this order) since $\hat \beta < z_a$. The continuity of $\tilde J_\alpha(\cdot)$ and \eqref{eq:firstBoundPsiLead} lead to \eqref{eq:EstimatePsilead}.
\end{proof}

\begin{proof}[Proof of \eqref{eq:psinegl}]
For $l_0\in \mathbb{Z}_{+}$,
define
\begin{equation}\label{eq:Ydefin}
    Y_{\xi,M,N}(l_0):= \sup_{x_0\in A_{l_0}} \sum_{k=1}^\infty \sigma_N^{2k}\sum_{\substack{\mathbf l \in \mathbb Z_+^k,(l_1,\dots,,l_{k-1})\in C_{k-1}\\l_k \text{ is bad}}} \sum_{\substack{x_i\in A_{l_i} \\ \forall i=1,\dots ,k}} \prod_{i=1}^k R_{N^\xi}(x_i-x_{i-1})h(x_i) \,,
\end{equation}
where we recall the definitions \eqref{eq:bad1} and \eqref{eq:bad2}.
 Then, inspired by \cite[Lemma 7.3]{CN25},
 we can obtain the following bound
\begin{equation}\label{eq:boundpsineglY}
    \Psi_{\xi,M,N}^{negl}\leq \sup_{l_0 \in \Z_{+}}\Psi_{\xi,M,N}^{lead}(l_0) \sum_{m=1}^\infty \left( \sup_{l_0 \in \Z_{+}} Y_{\xi,M,N}(l_0)\right)^m \,.
\end{equation}
We start by showing that \begin{equation}\label{eq:supfinito}
 \sup_{N,M \geq 1} \sup_{l_0 \in \Z_{+}}\Psi_{\xi,M,N}^{lead}(l_0) < \infty\,.
\end{equation}
To do so, we exploit
$$\sup_{l_0 \in \Z_{+}}\Psi_{\xi,M,N}^{lead}(l_0) = \max \Big\lbrace \sup_{l_0 \le {\xi M } }\Psi_{\xi,M,N}^{lead}(l_0) \,, \ \sup_{ l_0 \ge {\xi M + 1} }\Psi_{\xi,M,N}^{lead}(l_0) \, \Big\rbrace\,.$$
Since $\tilde J_\alpha(s) \leq 1$ for all $s\leq z_a$, by \eqref{eq:EstimatePsilead} the first argument is uniformly bounded by a constant.
Hence, we only need to bound $\Psi_{\xi,M,N}^{lead}(l_0)$ for $l_0 \ge \xi M +1$.

For $l_0\geq \xi M+1$, we bound $\Psi_{\xi,M,N}^{lead}(l_0)$ as in \eqref{eq:PSILEADbd2}, then apply the estimate \eqref{eq:order_of_RNx} in Lemma \ref{lem:order_of_RNx} as in \eqref{eq:PSILEADbd2} for $k=2,\ldots,k$. 
The only term requiring separate treatment is $R_{N^\xi}(x_1-x)$. 
Thanks to the assumption $l_1 \leq \xi M -2$ (recall \eqref{eq:defCk}), we can use \eqref{eq:order_RN_more_general} in Lemma \ref{lem:order_of_RNx} with $l=M\xi -2$ and then get
$R_{N^\xi}(x_1-x) \le 4 R_{N^\xi}(x_1)$. This way, we can bound
\begin{equation*}
    \sup_{ l_0 \ge \xi M + 1 }\Psi_{\xi,M,N}^{lead}(l_0) \le 4\,  \Psi_{\xi,M,N}^{lead}(0) \,,
\end{equation*}
where the right-hand side is uniformly bounded by  \eqref{eq:EstimatePsilead}. 
 This proves \eqref{eq:supfinito}.

To conclude, we need to show that for a suitable constant $C \in (0,\infty)$ we have
\begin{equation}
    \label{eq:stimaY}
    \limsup_{N\to\infty} \sup_{l_0 \in \Z_{+}} Y_{\xi,M,N}(l_0) \le 
    \frac{C}{M^{(1+a)\wedge 1}} 
    \,.
\end{equation}
The above estimate and \eqref{eq:supfinito} imply that, for $M$ sufficiently large, we can bound \eqref{eq:boundpsineglY} as
\begin{equation*}
    \limsup_{N\to\infty} \Psi_{\xi,M,N}^{negl} \leq \limsup_{N\to\infty}\sup_{l_0 \in \Z_{+}}\Psi_{\xi,M,N}^{lead}(l_0) \sum_{m=0}^\infty \left( \frac{C}{M^{(1+a)\wedge 1}}\right)^m \le \frac{\tilde C}{M^{(1+a)\wedge 1}}\,,
\end{equation*}
 for some finite constant $\tilde C > 0$.

We are only left to show \eqref{eq:stimaY}. 
We start by estimating the last sum in \eqref{eq:Ydefin}, given by
\begin{equation}\label{eq:estimateY}
\sigma_N^2 \sum_{l_k \text{ is bad }} \sum_{x_k \in A_{l_k}} R_{N^\xi}(x_k-x_{k-1})h(x_k).
\end{equation}
We will show that for all $M>0$, there exists $N_0(M)$ such that for all $N\geq N_0$, 
\begin{equation}\label{eq:bound_bad_l}
   \forall \, l'\leq  \xi M -2,\quad  \sigma_N^2 \sup_{y\in A_{l'}}\sum_{l \text{ is bad}} \sum_{x\in A_l}R_{N^\xi}(x-y)h(x) \leq \frac{C}{M^{(1+a)\wedge 1}},
\end{equation}
uniformly for all $l \in \mathbb{Z}_+$ satisfying either $(i)$ or $(ii)$ below (recall \eqref{eq:bad1} and \eqref{eq:bad2}):
\begin{equation}\label{eq:bad_l}
(i) \quad l=0 \text{ or } l \geq \xi M-1 , \qquad 
(ii) \quad |l-l'|\leq 1 \text{ and } l\leq \xi M -2.
\end{equation}
Then, \eqref{eq:bound_bad_l} gives an upper bound on \eqref{eq:estimateY}, yielding that
\begin{equation*}
\begin{split}
    Y_{\xi,M,N}(l_0) &\le \frac{C'}{M} \, \sum_{k=1}^\infty \sigma_N^{2(k-1)}\sum_{\substack{(l_1,\dots, l_{k-1})\in C_{k-1}}} \sum_{x_i\in A_{l_i},\, i=1,\dots,k} \prod_{i=1}^{k-1} R_{N^\xi}(x_i-x_{i-1})h(x_i) \\
    &\le \frac{C'}{M} \sup_{l_0 \in \Z_{+}} \psi^{lead}_{\xi,M,N}(l_0)\,,
    \end{split}
\end{equation*}
which implies in turn \eqref{eq:stimaY}. 
It remains to prove \eqref{eq:bound_bad_l}. 

We begin by considering the case $(i)$ in \eqref{eq:bad_l}. Fix $l'\in \mathbb{Z}_+$ such that $l'\leq \xi M -2$ and $y\in A_{l'}$. We have that
\begin{align} \nonumber
    &\quad \sigma_N^2 \sum_{l: \, (i) \text{ holds}} \, \sum_{x \in A_{l}} R_{N^\xi}(x-y)h(x)\\
    &\le \sigma_N^2 \sum_{x \in A_{\xi M -1}} R_{N^\xi}(x-y)h(x) + \sigma_N^2 \sum_{x \in A_{\xi M}} R_{N^\xi}(x-y)h(x)\label{eq:twoterms0}\\
    & \  + \sigma_N^2 {\sum_{l \geq \xi M +1}} \sum_{x \in A_{l}} R_{N^\xi}(x-y)h(x) + \sigma_N^2 \sum_{x\in A_0} R_{N^\xi}(x-y)h(x) \,.\label{eq:twoterms}
\end{align}
Regarding the first term in \eqref{eq:twoterms0}, we use that $R_N(x)\leq R_N(0)$ (this follows for example from $p_{2n}(x)\leq p_{2n}(0)$ by the Cauchy-Schwarz inequality), 
so that for $N$ sufficiently large,  it can be bounded using Lemma \ref{lem:bound_h} as 
\begin{equation}\label{eq:cM}
    \sigma_N^2 \, R_{N^\xi}(0) \sum_{x \in A_{\xi M-1}} h(x) \le C\, \sigma_N^2 \, \Big( \frac{\xi}{2}\Big)^a \frac{(\log N)^{a+2}}{M} \le \frac{C}{M}\,,
\end{equation}
for some finite constant $C >0$. 
A similar bound also works for the second term in \eqref{eq:twoterms0} and the second term in \eqref{eq:twoterms} (with the observation that in the latter case we have an additional factor $1/M^{a}$ if $a<0$). Furthermore, the first term in \eqref{eq:twoterms} is bounded above by
\begin{equation}\label{eq:boundbound}
    \begin{split}
        \sigma_N^2 \sum_{l\geq \xi M +1} \sum_{x \in A_l} R_{N^\xi}(x-y)h(x) \le C \Vert h \Vert_\infty \sum_{|x|\geq N^{(\xi + 1/M)/2}} R_N(x-y) \leq e^{-cN^{1/M}},
    \end{split}
\end{equation}
where the inequality uses \eqref{eq:ChernoffBound} and holds uniformly in $y\in A_{l'},l'\leq \xi M-2$. The right-hand side of \eqref{eq:boundbound} can be made arbitrarily small for large $N$, which concludes the proof of \eqref{eq:bound_bad_l} in case (i).

We now turn to prove \eqref{eq:bound_bad_l} in the case \eqref{eq:bad_l}-$(ii)$. For all $l'\leq \xi M -2$ and $y\in A_{l'}$, we have that
\begin{align*} \nonumber 
    & \quad \sigma_N^2 \sum_{l\leq \xi M -2: \, |l-l'|\leq 1} \sum_{x\in A_l} R_{N^\xi}(x-y)h(x)
    \le \sigma_N^2\, R_{N^\xi}(0) \sum_{l \le \xi M -2:\, |l-l'|\le 1} \sum_{x \in A_l}h(x) \\
    &\le C\sigma_N^2\,  \, \frac{(\log N)^{a+2}}{M} \sum_{l \le \xi M -2:\, |l-l'|\le 1} \bigg( \frac{l+1}{2M} \bigg)^a
    \le \frac{3 C}{M^{(1+a)\wedge 1}}\,.
\end{align*}
\end{proof}

\begin{proof}[Proof of \eqref{eq:finiteness2ndMoment}]
It follows from putting together  \eqref{eq:PhiToPsi}, \eqref{eq:decomposition}, \eqref{eq:psinegl} and \eqref{eq:supfinito}.
\end{proof}

\subsection{The lower bound}\label{subsec:lowerBound}
Following the previous section, we now present a lower bound on the second moment \eqref{eq:defPhi}.
\begin{lemma}
[Lower bound] \label{prop:lowerBound}
Let $0\leq \zeta\leq \xi\leq 1$ and recall \eqref{eq:Al}. Then, for all $\hat{\beta}\in [0, z_a)$,
    \begin{equation}\label{eq:lowerBound}
        \liminf_{\delta\to 0}\, \liminf_{N\to\infty}\inf_{x_0\in [N^{\zeta - \delta},N^{\zeta+\delta}]\cap \Z^2_{even}} \Phi_{N^\xi}(x_0)
        \geq \frac{\tilde J_{\alpha}\Big( \hat{\beta}\zeta^{\frac{a+2}{2}}\Big)}{\tilde J_{\alpha}\Big( \hat{\beta}\xi^{\frac{a+2}{2}}\Big)}\,.
    \end{equation}    
\end{lemma}
\begin{proof}
By  \eqref{eq:second_moment_chaos}  and the change of variables $u_i = n_{i}-n_{i-1}$ for $i\leq k$, we have, for all $K\in \mathbb N$,
\begin{equation*}
\begin{split}
    \Phi_{N^{\xi}}(x_0)&= 1+\sum_{k=1}^\infty \sigma_N^{2k} \sum_{\substack{\sum_{i=1}^k u_i \le N^\xi \\ u_i\geq 1, \,i=1,\dots,k}} \, \sum_{x_1,\ldots,x_k \in \Z^2_{\text{even}}} \prod_{i=1}^k p_{2u_i}(x_i-x_{i-1}) \, h(x_i)\\
    &\geq 1+ \sum_{k=1}^K \sigma_N^{2k} \sum_{\substack{\sum_{i=1}^k u_i \le N^\xi \\ u_i\geq 1,\,i=1,\dots,k}} \, \sum_{\substack{x_1,\ldots,x_k \in \Z^2_{\text{even}} \\ |x_i| \le N^{\frac{\xi}{2}(1 + \frac{1}{M})}}} \prod_{i=1}^k p_{2u_i}(x_i-x_{i-1}) \, h(x_i).
\end{split}
\end{equation*}
Observe that $u_i \leq N^{\xi}/K$  for all $i\leq k\leq K$ implies $\sum_{i\leq k} u_i \leq N^{\xi}$, hence for $\xi' \coloneqq \xi - \frac{\log K}{\log N}$,
\begin{align*}
\Psi_{N^{\xi}}(x_0) & \ge 1+ \sum_{k=1}^K \sigma_N^{2k} \sum_{\substack{ 1 \le u_i \le N^{\xi'} \\ i=1,\ldots,k}}\sum_{\substack{x_1,\ldots,x_k \in \Z^2_{\text{even}} \\ |x_i| \le N^{\frac{\xi}{2}(1 + \frac{1}{M})}}} \prod_{i=1}^k p_{2u_i}(x_i-x_{i-1}) \, h(x_i)\,\\
& = 1+\sum_{k=1}^K \sigma_N^{2k}  \, \sum_{\substack{x_1,\ldots,x_k \in \Z^2_{\text{even}} \\ |x_i| \le N^{\frac{\xi}{2}(1 + \frac{1}{M})}}} \prod_{i=1}^k R_{N^{\xi'}}(x_i-x_{i-1}) \, h(x_i).
\end{align*}
Let $M>0$ and recall $A_l$ in \eqref{eq:Al}. By Lemma \ref{lem:order_of_RNx} (first line of \eqref{eq:order_of_RNx_bis}), for $N \ge N_1=N_1(M)$ and $l_0=M\zeta$ or $l_0 = M\zeta - 1$, taking $M$ large enough so that $1/M\leq \delta$, we can further bound the last expression from below by
\begin{equation*}
    \begin{split}
        &\sum_{k=1}^K \sigma_N^{2k} \sum_{\substack{ 1\leq l_1,\ldots,l_k \le \lfloor \xi M \rfloor -1}} \, \sum_{ \substack{x_i \in A_{l_i} \\ i=1,\ldots,k}} \prod_{i=1}^k R_{N^{\xi'}}(x_i-x_{i-1}) \, h(x_i)\\
        & \ge \sum_{k=1}^K \sigma_N^{2k} \sum_{\substack{ 1\leq l_1,\ldots,l_k \le \lfloor \xi M \rfloor -1}} \, \sum_{ \substack{x_i \in A_{l_i} \\ i=1,\ldots,k}} \prod_{i=1}^k \bigg\lbrace \bigg( \xi' - \frac{l_i}{M} \vee \frac{l_{i-1}}{M} - \frac{4}{M} \bigg) \frac{\log N}{\pi } \bigg\rbrace \, h(x_i) \,.
    \end{split}
\end{equation*}
For any arbitrary $\varepsilon >0$, we obtain by Lemma \ref{lem:bound_h} that 
\begin{equation*}
\Phi_{N^{\xi}}(l_0) \geq \sum_{k=0}^K \bigg(\frac{\sigma_N^{2}   \, (1-\varepsilon) ( \log N )^{a+2}}{2M} \bigg)^k \sum_{\substack{ 1\leq l_1,\ldots,l_k \le \lfloor \xi M \rfloor -1}} \prod_{i=1}^k \bigg( \xi' - \frac{l_i}{M} \vee \frac{l_{i-1}}{M} - \frac{4}{M} \bigg) \bigg( \frac{l_i}{2M} \bigg)^a.
\end{equation*} 
Recalling \eqref{eq:beta_N} and setting $\sigma_N^2=\beta_N^2$, first send $N \to \infty$ (note that $\xi'\to \xi$ as $N\to\infty$ and that all the sums in the last display are finite), then $M \to \infty$ and $K \to \infty$ in this order, we find
\begin{equation*}
\begin{split}
    & 
    \liminf_{M\to\infty}\liminf_{N \to \infty} \inf_{\substack{x_0 \in A_{M\zeta}}} \mathbb{E} \big[ \Theta_\xi(x_0)^2\big]\\
    & \ge 1+ \sum_{k=1}^\infty (\bar{C}\,\hat \beta)^{2k}\int_{\substack{t_1,\dots, t_k\in [0,\xi] \\ t_0=\zeta}}\prod_{i=1}^k(\xi-t_i\vee t_{i-1})\,t_i^a \,\text{d}t_1\dots \text{d}t_k \\ \label{eq:upper_bound_estimate}
    &=1+\sum_{k=1}^\infty (\bar{C}\,\hat \beta)^{2k}\xi^{(2+a)k}\int_{\substack{s_1\dots s_k\in[0,1]\\ s_0=\zeta/\xi}}\prod_{i=1}^k(1-s_i\vee s_{i-1})\,s_i^a \, \text{d}s_1 \dots \text{d}s_k\,.
\end{split}
\end{equation*}
We then conclude similarly to the upper bound.
\end{proof}

\subsection{Proof of Theorem \ref{thm:main}} \label{subsec:proofOf2ndMoment}
The estimates \eqref{eq:secondMomentFlat} and \eqref{eq:2ndMomentSpaceVersion} follow from Lemma \ref{lem:second_moment}, Lemma \ref{thm:upper_bound} and Lemma \ref{prop:lowerBound}. Equation \eqref{eq:allsup2ndMoment} follows similarly by \eqref{eq:finiteness2ndMoment}.

It remains to prove \eqref{eq:2ndMomentZetaXilimit}, that we deduce from \eqref{eq:2ndMomentSpaceVersion}.  Let $\delta > 0$ and $\varepsilon>0$ that will be taken small compared to $\delta$.
Let us define $A_{\delta}:=\{z\in \mathbb{Z}_{\text{even}}^2: |z|\in[N^{(\zeta-\delta)/2},N^{(\zeta +\delta)/2}]\}$. 
By \eqref{eq:2ndMomentSpaceVersion}, it is enough to compare $\IE[\Theta^2_{\zeta,\xi}]$ with $\E[\Theta_{\xi}(x_0)\Theta_{\xi}(0)]$ for $x_0\in A_{\delta}$.
By \eqref{eq:second_moment} and the Markov Property, we have
\begin{align*}
    \E[\Theta_{\zeta,\xi}^2]&=\ E\left[e^{\beta_N^2\sum_{n=N^{\zeta}+1}^{N^\xi}   h(S_{2n})}\right]=\sum_{y\in\mathbb{Z}_{\text{even}}^2}p_{N^{\zeta}}(y) E_y\left[e^{ \beta_N^2\sum_{n=1}^{N^\xi-N^\zeta}  h(S_{2n})}\right] \\
    &\geq \left(\sum_{y\in A_{\delta}} p_{N^\zeta}(y) \right) \left(\inf_{y\in A_{\delta}}E_y \left[e^{ \beta_N^2\sum_{n=1}^{N^\xi-N^\zeta}  h(S_{2n})}\right]\right).
\end{align*}
By the central limit theorem applied to $S_N$, for $N\geq N_1(\delta,\varepsilon)$ large enough we have that $\sum_{y\in A_{\delta}} p_{N^\zeta}(y) \geq 1-\varepsilon$ since for all $\delta>0$,
\begin{equation*}
    \sum_{y\in A_{\delta}} p_{N^\zeta}(y) = P\left(\frac{|S_{N^\zeta}|}{N^{\zeta/2}}\in [N^{-\frac{\delta}{2}},N^{\frac{\delta}{2}}]\right) \xrightarrow[N\to\infty]{} 1\,.
\end{equation*}
Moreover, by \eqref{eq:2ndMomentSpaceVersion} and since $\zeta<\xi$, $N^\xi -N^\zeta=N^\xi(1+o(1))$, for  $\delta \leq \delta_1(\varepsilon)$ small enough and $N\geq N_2(\delta,\varepsilon)$ large enough,
\begin{equation*}
    \inf_{y\in A_{\delta}}E_y \left[e^{ \beta_N^2\sum_{n=1}^{N^\xi-N^\zeta}  h(S_{2n})}\right]\geq \frac{\tilde J_{\alpha}\Big( \hat{\beta}\zeta^{\frac{a+2}{2}}\Big)}{\tilde J_{\alpha}\Big( \hat{\beta}\xi^{\frac{a+2}{2}}\Big)}(1-\varepsilon),
\end{equation*}
thus yielding $\liminf_{N\to\infty} \E[\Theta_{\zeta,\xi}^2] \geq \tilde J_{\alpha}\Big( \hat{\beta}\zeta^{\frac{a+2}{2}}\Big){\tilde J_{\alpha}\Big( \hat{\beta}\xi^{\frac{a+2}{2}}\Big)}^{-1}$.

On the other hand, for any $\tilde{x}$ such that $N^{\zeta/2-\delta/2}\leq |\tilde{x}| \leq N^{\zeta/2 - \delta/4}$, we have
\begin{align} \nonumber
\sup_{x_0\in A_\delta}\E[\Theta_{\xi}(x_0)\Theta_{\xi}(0)]&\geq E_{\tilde{x}}\left[e^{\beta_N^2\sum_{n=1}^{N^\xi}   h(S_{2n})}\right]
\geq E_{\tilde{x}}\left[e^{\beta_N^2\sum_{n=N^{\zeta}+1}^{N^\xi}   h(S_{2n})}\right] \\  \label{eq:est_1} &=\sum_{y\in\mathbb{Z}^2}p_{N^\zeta}(y-\tilde{x})E_y\left[e^{ \beta_N^2\sum_{n=1}^{N^\xi-N^\zeta}  h(S_{2n})}\right],
\end{align}
where the last equality holds for the Markov Property.
Moreover, for all $\delta \leq \delta_2(\varepsilon)$, $N\geq N_3(\delta,\varepsilon)$ and $y\in A_{\delta/4}$, i.e.\ $|y|\in[N^{\zeta/2-\delta/8},N^{\zeta/2+\delta/8}]$, by the Cauchy-Schwartz inequality and the Local Limit Theorem,
\begin{align}\nonumber
    p_{N^\zeta}(y-\tilde{x})&\geq \frac{4}{N^\zeta}\frac{e^{-\frac{|y-\tilde{x}|^2}{N^\zeta}}}{2\pi}(1-\varepsilon)\geq \frac{4}{N^\zeta}\frac{e^{-\frac{|y|^2+2|y||\tilde{x}|+|\tilde{x}|^2}{N^\zeta}}}{2\pi}(1-\varepsilon)\\ \nonumber
    &\geq p_{N^\zeta}(y)e^{-(2N^{-\delta/8}+N^{-\delta/2})}(1-\varepsilon)\geq p_{N^\zeta}(y)(1-\varepsilon)^2,
\end{align}
hence \eqref{eq:est_1} yields
\begin{equation}\label{eq:est_2}
    \sup_{x_0\in A_\delta} \mathbb{E}[\Theta_\xi(x_0)\Theta_\xi(0)]\geq \sum_{y\in A_{\delta/8}} p_{N^\zeta}(y)(1-\varepsilon)^2 E_y\left[e^{ \beta_N^2\sum_{n=1}^{N^\xi-N^\zeta}  h(S_{2n})}\right].
\end{equation}
Now, for all $\varepsilon,\delta>0$, with $\delta$ small compared to $\varepsilon$ and $N\geq N_4(\delta,\varepsilon)$, the quantity $E_y\left[e^{ \beta_N^2\sum_{n=1}^{N^\xi-N^\zeta}  h(S_{2n})}\right]$ is bounded from above (cf.\@ \eqref{eq:allsup2ndMoment} and \eqref{eq:second_moment}) and $P(S_{N^\zeta}\not \in A_{\delta/4})\leq \varepsilon$ by the central limit theorem, which plugged into \eqref{eq:est_2} yields
\begin{align} \nonumber
  \sup_{x_0\in A_\delta} \mathbb{E}[\Theta_\xi(x_0)\Theta_\xi(0)] &\geq \sum_{y\in A_{\delta/8}} p_{N^\zeta}(y)(1-\varepsilon)^2 E_y\left[e^{ \beta_N^2\sum_{n=1}^{N^\xi-N^\zeta}  h(S_{2n})}\right]\\ &\nonumber
   \geq \sum_{y\in \mathbb{Z}^2} p_{N^\zeta}(y)(1-\varepsilon)^2 E_y\left[e^{ \beta_N^2\sum_{n=1}^{N^\xi-N^\zeta}  h(S_{2n})}\right] - c(1-\varepsilon)^3\varepsilon\\ \label{eq:est_3}
   &= \IE [\Theta_{\zeta,\xi}^2](1-\varepsilon)^2-c(1-\varepsilon)^3\varepsilon,
\end{align}
where the last equality is due to the Markov Property.
Finally, by \eqref{eq:est_3}, \eqref{eq:2ndMomentSpaceVersion} and the arbitrariness of $\varepsilon$, we conclude that $\limsup_{N\to\infty}\IE[\Theta_{\zeta,\xi}^2]\leq \tilde J_{\alpha}\Big( \hat{\beta}\zeta^{\frac{a+2}{2}}\Big){\tilde J_{\alpha}\Big( \hat{\beta}\xi^{\frac{a+2}{2}}\Big)}^{-1} $.

\qed
\begin{remark}
We note that the same argument yields that, for all $\hat \beta < z_a$,
    \begin{equation} \label{eq:linearTime}
    \forall \, 0\leq \zeta<\xi\leq 1,\quad \lim_{N\to\infty} E\left[\prod_{n=N^{\zeta}+1}^{N^{\xi}} \{ 1+\beta_N^2 h(S_{2n})\}\right] = \frac{\tilde J_{\alpha}\left( \hat{\beta}\zeta^{\frac{a+2}{2}}\right)}{\tilde J_{\alpha}\left( \hat{\beta}\,\xi^{\frac{a+2}{2}}\right)}\,.
    \end{equation}
\end{remark}

\subsection{Proof of Theorem \ref{th:ErdosTaylor}}\label{subsec:ErdosTaylorLaplace}
By \cite[Section 8]{Kent80}, the moment generating function of $\tau_{0,1}^{\alpha}$ satisfies $E[e^{t\cdot\tau_{0,1}^\alpha}] = \tilde{J}_\alpha(\sqrt{2t})^{-1}$ for all $t\in[0,z_a^2/2]$, with $z_a$ as in \eqref{eq:defza}. (See also \cite{Kent78} for similar expressions regarding the Laplace transform ($t\leq 0$).) As convergence of moment generating functions on an open interval implies convergence in distribution, the second equality in \eqref{eq:secondMomentFlat} yields the claim.

\section{The decoupling argument}\label{sec:decoupling}

This section is dedicated to proving point \ref{item:decoupling} in Theorem \ref{thm:general}, adapting the proof of Theorem 2.2 in \cite{CoscoDonadini}.
We recall that $Z_N:=W_N$ in Theorem \ref{thm:general} and that $Z_{k,M,N}$ is defined in \eqref{eq:defZkMN}. 
\begin{theorem}[Dyadic time decoupling]
  \label{th:step1}
For all integers $M>0$ we have 
\begin{equation}\label{eq:prop_step1}
W_N - \prod_{k=1}^{M} Z_{k,M,N} \xrightarrow[N \to \infty]{L^2} 0\,.
\end{equation}
\end{theorem}
\begin{remark}
We also refer to \cite[Lemma 6.2 and eq.\ (6.3)]{caravenna3} and \cite{CZ23} where similar decompositions have already appeared.
\end{remark}

\begin{proof}
By Appendix \ref{sec:approximation}, the linear approximation 
\[
Y_N:= E\left[\prod_{i=1}^N \{1+\beta_N \omega(n,S_n)\}\right],
\]
satisfies $ W_N-Y_N \xrightarrow[N \to \infty]{L^2} 0$. Hence, it is enough to show that
 $Y_N - \prod_{k=1}^{M} Z_{k,M,N} \xrightarrow[N \to \infty]{L^2} 0$. 
Recall that $t_k = \lceil N^{\frac{k}{M}}\rceil$. For simplicity we write $Z_k=Z_{k,M,N}$.  
Fix $M>0$ and consider for all $k\leq M$:
\begin{equation*}
    \tilde{Z}_k= E\left[\prod_{n=t_{k}+1}^{N} \left\{ 1+\beta_N\omega(n,S_n) \right\}\right].
\end{equation*}
Note in particular that $ Z_{k} \indep \tilde{Z}_k$ (recall \eqref{eq:defZkMN}).
We first prove that 
as $N\to\infty$, 
\begin{equation}\label{eq:L2_Z0}
  Y_N-Z_1\tilde{Z}_1\cvLdeux 0.
\end{equation}
Since $\IE[|Y_N-Z_1\tilde{Z}_1|^2]=\IE[Y_N^2]-2\IE[Y_NZ_1\tilde{Z}_1]+\IE[Z_1^2\tilde{Z}_1^2]$, by using
\eqref{eq:linearTime},
\begin{equation}
\lim_{N\to\infty }\IE[Y_N^2] = \frac{1}{\tilde J_{\alpha}\Big( \hat{\beta}\Big)} \label{eq:conv_WN_step1}\quad \text{and} \quad \lim_{N\to\infty } \IE[Z_1^2\tilde{Z}_1^2]=
\frac{1}{\tilde J_{\alpha}\Big( \frac{\hat{\beta}}{M^{\frac{a+2}{2}}}\Big)}\frac{\tilde J_{\alpha}\Big( \frac{\hat{\beta}}{M^{\frac{a+2}{2}}}\Big)}{\tilde J_{\alpha}\Big( \hat{\beta}\Big)}=\frac{1}{\tilde J_{\alpha}\Big( \hat{\beta}\Big)}.
\end{equation}
Thus, it suffices to prove that 
\begin{equation}\label{eq:star}
  \liminf_{N\to\infty}\IE[Y_NZ_1\tilde{Z}_1]\geq  \frac{1}{\tilde J_{\alpha}\Big( \hat{\beta}\Big)}.
\end{equation}

Considering $S^1,S^2$ and $S^3$ three independent copies of $S$, we set $l_n^{1,j}:= \beta_N^2 h(S_n^1-S_n^j)$ and $l_{2n}:=\beta_N^2h(S_{2n})$. We see that:
\begin{equation*}
  \begin{split}
   \IE[Y_NZ_1\tilde{Z}_1]&=E^{\otimes 3}\left[ \prod_{n=1}^{t_1}\{1+l_n^{1,2}\} \prod_{n=t_1+1}^N\{1+l_n^{1,3}\}\right]\\
   &= E^{\otimes 3}\left[\prod_{n=1}^{t_1}\{1+l_n^{1,2}\} \,E_{S_{t_1}^1,S_{t_1}^3}^{\otimes 2}\left[ \prod_{n=t_1+1}^N\{1+ l_n^{1,3}\}\right]\right] \\
   &\geq E^{\otimes 3}\left[ \prod_{n=1}^{t_1}\{1+l_n^{1,2}\} \mathds{1}_{S_{t_1}^1-S_{t_1}^3\in A_N} \, E_{S_{t_1}^1-S_{t_1}^3}\left[\prod_{n=1}^{N-t_1}\{1+l_{2n}\}\right]\right],
 \end{split}
\end{equation*}
where we have used Markov’s property in the second line, and where we have set
\[
A_N=\{x\in\mathbb{Z}^2_{even} : \rho \sqrt{t_1} \leq |x| \leq \rho^{-1}\sqrt{t_1}\}\,,
\]
with $\rho>0$. Hence,
we obtain that
$\IE[Y_N Z_1 \tilde{Z}_1]\geq \psi_N\varphi_N$,
where
\begin{equation*}
\psi_N=E^{\otimes 3}\left[ \prod_{n=1}^{t_1}\{1+l_n^{1,2}\}\mathds{1}_{S_{t_1}^1-S_{t_1}^3\in A_N}\right] \qquad \text{and} \qquad \varphi_N=\inf_{x\in A_N}E_x\left[\prod_{n=1}^{N-t_1}\{1+l_{2n}\}\right].
\end{equation*}
Therefore, the inequality \eqref{eq:star} follows if we can prove:
\begin{equation}\label{eq:liminf_psi}
\liminf_{\rho\to 0}\liminf_{N\to \infty}\psi_N \geq \frac{1}{\tilde J_{\alpha}\Big( \frac{\hat{\beta}}{M^{\frac{a+2}{2}}}\Big)},\quad \text{and} \quad \liminf_{\rho\to 0}\liminf_{N\to\infty}\varphi_N\geq \frac{\tilde J_{\alpha}\Big( \frac{\hat{\beta}}{M^{\frac{a+2}{2}}}\Big)}{\tilde J_{\alpha}\Big( \hat{\beta}\Big)}.
\end{equation}
We start by proving the first estimate with $\psi_N$. 
We have
\begin{equation}\label{eq:psi}
\psi_N=E^{\otimes 2}\left[ \prod_{n=1}^{t_1}\{1+l_n^{1,2}\}\right]-E^{\otimes 3}\left[ \prod_{n=1}^{t_1}\{1+l_n^{1,2}\}\mathds{1}_{S_{t_1}^1-S_{t_1}^3\not\in A_N}\right]\,,
\end{equation}
where by \eqref{eq:linearTime},
$\lim_{N\to\infty}E^{\otimes 2}\left[ \prod_{n=1}^{t_1}\{1+l_n^{1,2}\}\right] = \frac{1}{\tilde J_{\alpha}\Big( \hat{\beta}/M^{\frac{a+2}{2}}\Big)}$.
We thus focus on the second term on the right-hand side of \eqref{eq:psi}. By H\"older Inequality with $p,q\geq 1$ s.t. $\frac{1}{p}+\frac{1}{q}=1$, we have
\begin{equation} \label{eq:EPq}
E^{\otimes 3}\left[ \prod_{n=1}^{t_1}\{1+ l_n^{1,2}\}\mathds{1}_{S_{t_1}^1-S_{t_1}^3\not\in A_N}\right]\leq E^{\otimes 2}\left[\prod_{n=1}^{t_1}\{1+ l_n^{1,2}\}^p\right]^{\frac{1}{p}}P^{\otimes 2}\left(S_{t_1}^1-S_{t_1}^3\not\in A_N\right)^{\frac{1}{q}}.
\end{equation}
Using the central limit theorem for sums of i.i.d.\ random variables yields
 \begin{equation} \label{eq:Pvanishes}
     P^{\otimes 2}\left(S_{t_1}^1-S_{t_1}^3\not\in A_N\right)^{\frac{1}{q}}=P\left(S_{2t_1}\not\in A_N\right)^{\frac{1}{q}}=P\left(\frac{S_{2t_1}}{\sqrt{2t_1}}\not\in \Big[\frac{\rho}{\sqrt{2}},\frac{\rho^{-1}}{\sqrt{2}}\Big]\right)^{\frac{1}{q}}\xrightarrow[\substack{N\to\infty \\ \rho\to 0}]{} 0,
 \end{equation}
 where we took first $N\to \infty$ and then $\rho \to 0$.
 Thus, by $(1+l_n^{1,2})^p\leq e^{p\,l_n^{1,2}}$ and \eqref{eq:finiteness2ndMoment} with $\hat \beta$ replaced by $\sqrt{p} \hat \beta <1$ for $p=p(\hat\beta)> 1$ small enough, the right-hand side of \eqref{eq:EPq} is uniformly bounded by a constant.  Combined with \eqref{eq:Pvanishes}, we obtain the first inequality of \eqref{eq:liminf_psi}.
On the other hand, the second lower bound of
\eqref{eq:liminf_psi} follows from Lemma \ref{prop:lowerBound}.

Repeating the proof of \eqref{eq:L2_Z0}, one can show that $\forall k< M$ we have
$\tilde{Z}_{k} - Z_{k+1}\tilde{Z}_{k+1} \cvLdeux 0$.
Then, by independence of $\tilde{Z}_{k}-Z_{k+1}\tilde{Z}_{k+1}$ and $(Z_i)_{i\leq k}$ and since $\limsup_N\IE[Z_i^2]<\infty$ by \eqref{eq:finiteness2ndMoment}, one obtains
\begin{equation*}
W_N-\prod_{k=1}^{M}Z_{k}
=W_N-Z_1\tilde{Z}_1+Z_1(\tilde{Z}_1-Z_2\tilde{Z_2}) +\dots +\prod_{i=1}^{M-2}Z_i(\tilde{Z}_{M-1}-Z_{M-1}\tilde{Z}_{M-1})\xrightarrow[N \to \infty]{L^2} 0.
\end{equation*}
\end{proof}

\section{Estimates on the increments} \label{sec:increments}
In the following, we provide estimates on the moments of the centered variables 
\[{U_{k,M,N} \coloneqq Z_{k,M,N} - 1}, \quad k = 1, \dots, M,\]
with $Z_{k,M,N}$ in \eqref{eq:defZkMN}. These estimates allow us to verify the remaining conditions \ref{item:secmomCLT} and \ref{item:Lindeberg}
of Theorem \ref{thm:general} and thereby conclude the proof of Theorem \ref{thm:clt}.
We emphasize that the variables \( U_{k,M,N} \)
are independent and centered.

\begin{lemma}[Second moment estimates]\label{lem:second_moment_Uk}
    For any $\hat{\beta} \in [0,z_a)$, there exists $C_{\hat \beta}>0$ such that
    \begin{equation}\label{eq:secmomUk}
    \limsup_{N\to\infty} \sup_{1 \le k\leq M}\IE\left[ U_{k,M,N}^2\right] \leq  \frac{C_{\hat \beta}}{M}\,,
\end{equation}
and
\begin{equation}
    \label{eq:secmomconvCLTTT}
    \lim_{M \to \infty} \, \lim_{N \to \infty} \sum_{k=1}^M \mathbb{E} \big[ U_{k,M,N}^2\big] = \lambda^2 \coloneqq \log \tilde J_\alpha (\hat \beta )^{-1}\,.
\end{equation}
\end{lemma}
Note that \eqref{eq:secmomconvCLTTT} directly implies condition \ref{item:secmomCLT} in Theorem \ref{thm:general}, where the value $\sigma_M^2$ therein can be explicitly derived in the following proof, while $\sigma^2=\lambda^2$. 

\begin{proof}
Fix $1 \le k \le M$ and recall that $\IE[U_{k,M,N}^2]=\IE[Z_{k,M,N}^2]-1$. By Lemma \ref{thm:upper_bound} and Appendix \ref{sec:approximation}, we easily get 
\begin{equation*}
\begin{split}
    \limsup_{N \to \infty} \sup_{1 \le k \le M} \mathbb{E} \big[ U_{k,M,N}^2 \big] &\le \sup_{1 \le k \le M}  \left\{\frac{\tilde{J}_{\alpha}\left(\hat{\beta} \big(\frac{k-1}{M}\big)^{\frac{a+2}{2}}\right)}{\tilde{J}_{\alpha}\left(\hat{\beta} \big(\frac{k}{M}\big)^{\frac{a+2}{2}}\right)}-1\right\}\\
    &= \sup_{1 \le k \le M}  \left\{\frac{\tilde{J}_{\alpha}\left(\hat{\beta} \big(\frac{k-1}{M}\big)^{\frac{a+2}{2}}\right)-\tilde{J}_{\alpha}\left(\hat{\beta} \big(\frac{k}{M}\big)^{\frac{a+2}{2}}\right)}{\tilde{J}_{\alpha}\left(\hat{\beta} \big(\frac{k}{M}\big)^{\frac{a+2}{2}}\right)}\right\}\\
    &\le \frac{1}{M }\frac{ \sup_{z \in [0,{\hat{\beta}}]} |\tilde{J}_\alpha'(z)|}{\tilde{J}_\alpha(\hat \beta )} =:\frac{C_{\hat \beta}}{M}\,,
\end{split}
\end{equation*}
where we applied the mean value theorem.

We now prove \eqref{eq:secmomconvCLTTT}. Again by Theorem \ref{thm:main}, Appendix \ref{sec:approximation} and the mean value theorem, we get
\begin{equation*}
    \begin{split}
       \sigma_M^2 \coloneqq \lim_{N \to \infty} \sum_{k=1}^M \mathbb{E} \big[ U_{k,M,N}^2 \big] &=   \sum_{k=1}^{M} \left\{\frac{\tilde{J}_{\alpha}\left(\hat{\beta} \big(\frac{k-1}{M}\big)^{\frac{a+2}{2}}\right)-\tilde{J}_{\alpha}\left(\hat{\beta} \big(\frac{k}{M}\big)^{\frac{a+2}{2}}\right)}{\tilde{J}_{\alpha}\left(\hat{\beta} \big(\frac{k}{M}\big)^{\frac{a+2}{2}}\right)}\right\}\\
       &= -\frac{1}{M} \sum_{k=1}^{M} \frac{\tilde{J}'_{\alpha}\left(\hat{\beta} \big(z_k\big)^{\frac{a+2}{2}}\right)}{\tilde{J}_{\alpha}\left(\hat{\beta} \big(\frac{k}{M}\big)^{\frac{a+2}{2}}\right)}\,,
    \end{split}
\end{equation*}
for some $z_k \in \left( \frac{k-1}{M}\,, \frac{k}{M} \right)$, $k=1,\ldots,M$. By a Riemann sum approximation, sending $M \to \infty$ we finally obtain the conclusion, indeed
\begin{equation*}
    \lim_{M \to \infty} \sigma_M^2 = - \int_0^1 \frac{\tilde{J}'_{\alpha}\left(\hat{\beta} \big(t\big)^{\frac{a+2}{2}}\right)}{\tilde{J}_{\alpha}\left(\hat{\beta} \big(t\big)^{\frac{a+2}{2}}\right)} \, \text{d}t = - \log \tilde J_\alpha (\hat \beta ) =\log \tilde J_\alpha (\hat \beta )^{-1}\,,
\end{equation*}
where we have used that $\tilde J_\alpha(0)=1$ for all $\alpha > -1$.
\end{proof}

In the proof of the following result, we show that $U_{k,M,N}$ can be naturally expressed as a \emph{polynomial chaos} in the variables $\omega$'s. This will be essential to exploit \emph{hypercontractivity of polynomial chaos} (see \cite{caravenna2, MOO10} and \cite{janson_1997}) and therefore control the moments of $U_{k,M,N}$ slightly larger than two in terms of the second moment bound \eqref{eq:secmomUk}. However, due to the correlation in the environment $\omega$, classical hypercontractivity results — which rely on independence — cannot be applied directly. To overcome this, we express $\omega$ as a sum of suitably chosen independent Gaussian variables, allowing us to represent $U_{k,M,N}$ as a polynomial chaos in independent variables.

\begin{lemma}[Higher moments estimate]\label{lemma:fourth_moment}
There exist $\varepsilon_0=\varepsilon_0(\hat \beta)\in (0,1)$ and $c_{\hat \beta} <\infty$ such that for all $M>0$:
\begin{equation}\label{eq:fourth_moment}
  \limsup_{N\to\infty} \sup_{1 \le k\leq M} \IE[\left|U_{k,M,N}\right|^{2+\varepsilon_0}]\leq \frac{c_{\hat \beta}}{M^{1+\frac{\varepsilon_0}{2}}}.
\end{equation}
\end{lemma}

We stress that \eqref{eq:fourth_moment} implies \eqref{eq:Lindeberg}, thus condition \ref{item:Lindeberg} holds true.

\begin{proof}
By definition of $Z_{k,M,N}$ (see \eqref{eq:defZkMN}), $U_{k,M,N}=Z_{k,M,N}-1$ can be expressed as a polynomial chaos, i.e.\ a multilinear polynomial in the variables $\omega(n,x)$'s: 
\begin{equation}\label{eq:Ukpolch}
    U_{k,M,N}=\sum_{m=1}^\infty \beta_N^m \sum_{\substack{t_{k-1}<n_1<\cdots<n_m \le t_k\\ x_1,\ldots,x_m \in \Z^2, \ x_0:= 0}} \prod_{i=1}^m p_{n_i-n_{i-1}}(x_i-x_{i-1}) \, \omega(n_i,x_i)\,.
\end{equation}
Since the disorder is Gaussian by assumption, for any $(n,x)\in \N \times \Z^2$ it is possible to write $\omega(n,x) = \sum_{y \in \Z^2} h_0(y) \omega_0(n,y+x)$, where $\{\omega_0(n,x)\}_{(n,x)\in \N \times \Z^2}$ are i.i.d.\ standard normal random variables and $h_0$ was introduced below the expression \eqref{eq:correlation_function}\footnote{Indeed, notice that for all $(n,x), (n,x') \in \N \times \Z^2$, we have
$\mathbb{E} [ \omega(n,x)\omega(n,x')]
        =\sum_{y \in \Z^2}h_0(y) h_0(y+x-x')= h_0 * h_0 (x-x')=h(x-x')\,.$}. 
Plugging the new expression into \eqref{eq:Ukpolch} and by the change of variables $z_i=y_i+x_i$ for $i=1,\ldots,m$, we write $U_{k,M,N}=\sum_{m=0}^\infty X_m^{(N)}$, where $X_0^{(N)}=0$ and for $m \ge 1$ each $X_m^{(N)}$ is a polynomial chaos in the \emph{independent} variables $\omega_0(n,x)$'s:
\begin{equation} \label{eq:X_mDecU_N}
    X_m^{(N)}= \beta_N^m \sum_{\substack{t_{k-1}<n_1<\cdots<n_m \le t_k\\ z_1,\ldots,z_m \in \Z^2}}  \, \psi\big( n_1,z_1,\ldots,n_m,z_m\big) \prod_{i=1}^m \omega_0(n_i,z_i)\,,
\end{equation}
with 
\[
\psi\left( n_1,z_1,\ldots,n_m,z_m\right) := \sum_{y_1,\ldots,y_m \in \Z^2,y_0:=0 } \prod_{i=1}^m p_{n_i-n_{i-1}}\left( (z_i-z_{i-1})-(y_i-y_{i-1})\right)h_0(y_i).
\]
By the Gaussianity of the $\omega_0(n,x)$'s and by Proposition \ref{thm:upper_bound}, the formulas (B.3) and (B.4) in \cite{caravenna2} hold true. We can then apply \emph{hypercontractivity for polynomial chaos} \cite[Theorem B.1]{caravenna2} and deduce that for every $\varepsilon >0$ there exists a constant $c_\varepsilon$ uniform in $N$ such that $c_\varepsilon\to 1$ as $\varepsilon\to 0$ (in particular, $c_\varepsilon=\sqrt{1+\varepsilon}$ in the Gaussian case, see \cite{MOO10}) and 
\begin{equation*}
\begin{split}
\IE\left[ |U_{k,M,N}|^{2+\varepsilon} \right] &= \IE \left[ \left| \sum_{m=1}^\infty X_m^{(N)} \right|^{2+\varepsilon}\right]\leq \left( \sum_{m=1}^\infty c_\varepsilon^{2m}\IE\left[ (X_m^{(N)})^2 \right] \right)^{1+\frac{\varepsilon}{2}}\\
&=\left(\sum_{m=1}^\infty c_\varepsilon^{2m}\beta_N^{2m}  \sum_{\substack{t_{k-1}<n_1<\cdots<n_m \le t_k\\ z_1,\ldots,z_m \in \Z^2}}\, \psi\big( n_1,z_1,\ldots,n_m,z_m\big)^2 \right)^{1+\frac{\varepsilon}{2}}.
\end{split}
\end{equation*}
Observe that the expression above corresponds to the second moment of $U_{k,M,N}$ where instead of $\hat{\beta}$ we consider $c_\varepsilon\,\hat{\beta}$ (recall definitions \eqref{eq:defZkMN} and \eqref{eq:beta_N}). (This can be seen from the expression \eqref{eq:X_mDecU_N} since $U_{k,M,N}=\sum_{m=0}^\infty X_m^{(N)}$ is a sum of orthogonal terms in $L^2$.) Hence, since $\hat{\beta} < z_a$, we can choose $\varepsilon_0=\varepsilon_0(\hat\beta)\in (0,1)$ such that $c_{\varepsilon_0} \, \hat\beta < z_a$ and conclude by \eqref{eq:secmomUk} that
\begin{equation*}
\IE\left[ |U_{k,M,N}|^{2+\varepsilon_0} \right] 
\leq \frac{c_{\hat\beta}}{M^{1+\frac{\varepsilon_0}{2}}}\,,
\end{equation*}
for $N$ large enough, uniformly in $1 \le k\leq M$.
\end{proof}

\section{Proof of Theorem \ref{thm:clt} }\label{sec:end_of_proof}
For $\hat{\beta} \in [0,z_a)$, the central limit theorem \eqref{eq:clt} for the log-partition function $\log W_N^{\beta_N}$ follows by Theorem \ref{thm:general}, whose assumptions have been verified in the previous sections.

The only remaining step is to prove that $W_N^{\beta_N}$ converges to $0$ in probability for $\hat{\beta} \ge z_a$. The argument is standard and can be recovered by following the proof of Theorem 2.8 in \cite[Section 6]{caravenna1}. In particular, it suffices to show that $\mathbb{E}[(W_N^{\beta})^\theta]$ is non-increasing in $\beta$ for some $\theta \in (0,1)$. We note that the proof in our framework is exactly the same as in \cite[Section 6]{caravenna1}. The only delicate point is the application of the FKG inequality in the case of disorder variables that are not independent. However, the inequality still holds under our assumptions, since the variables $\omega$ we consider can be written as linear combinations of independent Gaussian variables with positive coefficients (recall the proof of Lemma \ref{lemma:fourth_moment}).

\section{A recursive sequence of polynomials}\label{sec:recursive_formula}
We prove the recursive formula at the heart of the limit in \eqref{eq:2ndMomentZetaXilimit}.

Let $f_0\equiv 1$ and define the sequence of polynomials $f_k:[0,1]\to \mathbb R_+$ as follows. For all $x\in [0,1]$,
\begin{equation} \label{eq:recursiveformula}
\begin{aligned}
    f_{k+1}(x) &\coloneqq 
    \int_{0}^{1} (1-t \vee x )\, t^a f_k(t) \, \text{d}t=
    (1-x) \int_{0}^{x} t^a f_k(t) \, \text{d}t +  \int_{x}^{1} (1-t)\, t^a f_k(t) \, \text{d}t.
\end{aligned}
\end{equation}
\begin{remark} \label{rk:f''}
Alternatively, \eqref{eq:recursiveformula} writes $f''_{k+1}(x)=-x^a f_k(x)$, $f_{k+1}'(0)=0$, $f_{k+1}(1)=0$.
\end{remark}
Set $\alpha = \frac{a+1}{a+2}-1$ and define $\tilde J_\alpha,\,z_a$ as above \eqref{eq:defJalpha}.
\begin{theorem} \label{lem:recursiveLemmaFormula}
For all $a>-1$, $z\in \big[\,0,\,(a+2)z_a/2\,\big)$ and $x\in [0,1]$,
\begin{equation} \label{eq:keyFormula0}
    \begin{split}
\sum_{k=0}^\infty   f_k(x) z^{2k}  =   \frac{ \tilde J_{\alpha}\left(\frac{2}{a+2}\,z  x^{\frac{a+2}{2}} \right)}{\tilde J_{\alpha}\left(\frac{2}{a+2}\,z\right)}.
\end{split}
\end{equation}
\end{theorem}

\begin{proof}
By Remark \ref{rk:f''}, one readily sees that $f_k$ is a polynomial in $x$ of the form
\begin{equation}\label{eq:recursivef}
    f_k(x) = \sum_{n=0}^{k} c_n^{(k)} x^{n(a+2)}\,,
\end{equation}
with $c_n^{(k)}\in \mathbb R$ for $n\leq k$.
Introducing the operator
\begin{equation*}
    F(g)(x):= \, \bigg(  (1-x) \int_{0}^{x} t^a g(t) \, \text{d}t +  \int_{x}^{1} (1-t)\, t^a g(t) \, \text{d}t\bigg)\,,
\end{equation*} 
and setting $g_n(x):= x^{n(a+2)}$, we easily have
\begin{equation*}
\begin{split}
    F(g_n)(x) &= \bigg( (1-x) \int_{0}^{x} t^a \, t^{n(a+2)} \, \text{d}t +  \int_{x}^{1} (1-t)\, t^a \, t^{n(a+2)} \, \text{d}t \bigg)=\lambda_n^a \big( 1-x^{(n+1)(a+2)} \big)\,,
\end{split}
\end{equation*}
with $\lambda_n^a := \Big( (n+1)(a+2)\big((n+1)(a+2)-1\big)\Big)^{-1}$. Therefore, as $f_{k+1}(x)=F(f_k)(x)$,
\begin{equation}\label{eq:k-k+1}
    \begin{split}
        f_{k+1}(x)= \sum_{n=0}^k c_n^{(k)} F(g_n)(x)&=\sum_{n=0}^k c_n^{(k)}\, \lambda_n^a \, \big( 1-x^{(n+1)(a+2)} \big)= \sum_{n=0}^{k+1} c_n^{(k+1)} x^{n(a+2)}\,,
    \end{split}
\end{equation}
so that for all $n \in\{0,\dots,k+1\}$,
\begin{equation}\label{eq:recursiveC}
\begin{split}
    c_{n}^{(k+1)}&=-\lambda_{n-1}^a \, c_{n-1}^{(k)}=\lambda_{n-1}^a \, \lambda_{n-2}^a \, c_{n-2}^{(k-1)} = \cdots  = (-1)^{n}  \psi^a(n) \, c_0^{(k+1-n)}\, ,
    \end{split}
\end{equation}
where $\psi^a(0)=1$ and denoting by $\Gamma$ the Gamma function, using $\Gamma(z+1)=z\,\Gamma(z)$,
\begin{equation}
    \label{eq:psia}
\psi^a(n) := \prod_{i=0}^{n-1} \lambda_i^a = \frac{1}{n!}\frac{\Gamma(\alpha')}{\Gamma(n+\alpha')}\frac{1}{(a+2)^{2n}}, \quad \alpha' \coloneqq \frac{a+1}{a+2}\,.
\end{equation}
Since $f_{k+1}(1)=\sum_{n=0}^{k+1} c_n^{(k+1)}=0$ (recall \eqref{eq:recursiveformula} and \eqref{eq:recursivef}), from \eqref{eq:recursiveC} we obtain
\begin{equation}\label{eq:C0}
    c_0^{(k+1)}=-\sum_{n=1}^{k+1}c_n^{(k+1)}=\sum_{n=0}^{k} (-1)^{n} \psi^a(n+1) \, c_0^{(k-n)}\,.
\end{equation}

Next, define the generating functions
\begin{equation}\label{eq:A(y)}
    A(z)=\sum_{k=0}^\infty c_0^{(k)} z^{2k},\quad B(z,x) = \sum_{k=0}^\infty f_k(x) z^{2k},\quad z\in \big[0,R(x)\big), \ x\in [0,1]\,,
\end{equation}
where $c_0^{(k)}=f_k(0)$ satisfies the recursive formula \eqref{eq:C0} and $R(x)$ is the radius of convergence of $B(z,x)$ (and $R(0)$ is the radius of convergence of $A(z)=B(z,0)$). We note that $R(x)\geq 1$ for all $x\in [0,1]$ since $f_k(x)\in  [0,1]$.
Recalling \eqref{eq:recursivef}-\eqref{eq:recursiveC}, for all $z \in \big[0, R(x)\wedge R(0)\big)$,
\begin{equation}\label{eq:GA}
    \begin{split}
        B(z,x) & = \sum_{k=0}^\infty \sum_{n=0}^k (-1)^n  \psi^a(n)  c_0^{(k-n)} x^{n(a+2)}z^{2k} \\
        &=\sum_{n=0}^\infty (-1)^n \psi^a(n) x^{n(a+2)}  z^{2n}\sum_{k=n}^\infty c_0^{(k-n)} z^{2k-2n}= G(x,z) \, A(z)\,,
    \end{split}
\end{equation}
where we denoted
$G(x,z) \coloneqq \sum_{n=0}^\infty (-1)^n  \psi^a(n) x^{n(a+2)} z^{2n}$, whose radius of convergence is infinite.
We now compute $A(z)$ for $z\in (0,R(0))$.
As $c^{(0)}_0=1$, we obtain by \eqref{eq:C0},
\begin{equation*}
    \begin{split}
        A(z) = 1 + \sum_{k=0}^\infty c_0^{(k+1)} \,z^{2(k+1)} & =1+ \sum_{k=0}^\infty \sum_{n=0}^k (-1)^{n} \psi^a(n+1)  c_0^{(k-n)}  z^{2(k-n)} \, z^{2(n+1)}\\
        &=1+A(z)  \sum_{n=0}^\infty (-1)^{n}  \psi^a(n+1)  z^{2(n+1)} \,.
    \end{split}
\end{equation*}
Recalling \eqref{eq:psia}, we thus find (when the inverse is well defined)
\begin{equation}
\begin{aligned}
    A(z) &= \bigg( 1-\sum_{n=0}^\infty (-1)^{n}  \psi^a(n+1)  z^{2(n+1)} \bigg)^{-1} = \bigg( \sum_{n=0}^\infty (-1)^n  \psi^a(n)   z^{2n} \bigg)^{-1}\\
    &= \bigg( \Gamma(\alpha')\sum_{n=0}^\infty (-1)^n \,\frac{1}{n!} \, \frac{1}{\Gamma(n+\alpha')} \,\bigg( \frac{z}{a+2}\bigg)^{2n}\, \bigg)^{-1} = \tilde J_\alpha \left( \frac{2z}{a+2}\,\right)^{-1}, 
\end{aligned}
\label{eq:A(y)explicit}
\end{equation}
with $\alpha = \alpha'-1 = \frac{a+1}{a+2}-1$.
For $\alpha > -1$ (which holds for $a>-1$), it is known that all zeros of the Bessel function $J_\alpha$ are real and thus $\tilde J_\alpha(z)\neq 0$ for all $|z| <z_a$ on the complex plane, where we recall that $z_a$ is the first real zero of $\tilde J_\alpha$. Therefore,  \eqref{eq:A(y)explicit} holds for all $z\in \big[0,(a+2)z_a/2\big)$ by analytic continuation.

By a similar computation,  for all $z>0$,
\begin{equation*}
    \begin{split}
         G(x,z) = \sum_{n=0}^\infty (-1)^n  \psi^a(n) x^{n(a+2)} z^{2n}
         = \tilde  J_{\alpha}\bigg( \frac{2 x^{(a+2)/2}\,z}{a+2}\,\bigg)\,,
    \end{split}
\end{equation*} 
hence by \eqref{eq:GA}, 
$
    B(z,x) =  {\tilde J_{\alpha}\left( \frac{2\sqrt{x^{a+2}\,z}}{a+2}\right)}{\tilde J_{\alpha}\left( \frac{2\sqrt{z}}{a+2}\,\right)}^{-1},
$
where the identity can be extended to $z\in [0,z_a)$.
\end{proof}

\section{Concluding remarks and open questions}
\label{sec:openQuestions}
We list below some final comments and possible open problems.
\begin{itemize} 
\item
For a critical spatial correlation decaying as $h(x) \sim |x|^{-2}$, Lacoin \cite{Lacoin-cor} gave a conjecture about the asymptotic behavior of the free energy $p(\beta)$ defined by
\[p(\beta) := \lim_{N\to\infty} \frac{1}{N} \log W_N^\beta, \quad  \ W_N^\beta:= \frac{Z_N^\beta}{\IE[Z_N^\beta]}.
\]
When $d=3$ and $\alpha = 2$, he expected that $p(\beta) = \exp\{-\frac{c}{\beta^2}(1+o(1))\}$ as  $\beta\to 0$, while $p(\beta) = \exp\{-\frac{c'}{\beta}(1+o(1))\}$ for $d=2$ and $\alpha = 2$. 
With a finer tuning $h(x)\sim (\log |x|)^a/|x|^2$ as in \eqref{eq:correlation_function}, one can expect that when $d=2$ and $a>-1$,
\begin{equation} \tag{conj} \label{eq:freeEnergy}
p(\beta)= \exp\left\{-\left(\frac{c_{a}}{\beta}\right)^{\frac{2}{a+2}}(1+o(1))\right\}, \quad \beta\to 0.
\end{equation}
Theorem \ref{thm:main} hints that $c_a=\mathfrak C_a z_a$, where   $z_a$ and $\mathfrak C_a$  are defined in  \eqref{eq:defza} and \eqref{eq:beta_N}. (In fact, by analogy with the space-time independent case, the proof of Theorem \ref{thm:main} should provide some upper bound on the free energy matching \eqref{eq:freeEnergy}, while  the lower bound may require rather involved techniques, see \cite{BH17,BCT25} for more details. We note that in the latter reference, sharp estimates on the error terms $o(1)$ have been recently derived.)
\item What happens when $a=-1$ in \eqref{eq:correlation_function}? We stress that for $a<-1$, the function $h$ is integrable, which implies that the polymer should behave identically to space-time independent weights, cf.\@ \cite{Lacoin-cor}. 
\item We did not touch the cases $d\geq 3$ and $\alpha = 2$. We believe that an induction procedure in the spirit of \eqref{eq:firstBoundPsiLead}, \eqref{eq:recursiveformula} may appear, but we do not expect a phase transition in $\hat \beta$ nor special functions to arise. We refer to \cite{MuTri04} for existing results on this regime in the continuum setting, although we emphasize that the potential considered by Mueller and Tribe is, unlike ours, singular at the origin.
To our understanding, this features creates a blowup at low temperature that we do not expect in our setup. %We believe that this discrepancy is related to the difference in the limiting processes mentioned in the discussion below Theorem \ref{th:ErdosTaylor} when $h$ is bounded or not around zero.

\item 
    Assume $\beta_N$ as in \eqref{eq:beta_N}. At the \emph{critical point} $\hat \beta = \hat \beta_c$, we conjecture that the properly rescaled partition function converges to a non-trivial limit. By analogy to the 
    independent noise case, we expect this limit to be the analogue of the celebrated Stochastic Heat Flow \cite{CaSuZyCrit21,Nak25,TsaiMoments24} for spatially correlated driving noise. This is natural to believe in view of \cite{CaSuZyCrit21}, where the Stochastic Heat Flow was first constructed by taking the diffusive limit of the (space-time independent) two-dimensional directed polymer at some critical temperature $\beta_c$, for $\beta_N \sim  \beta / \sqrt{\log N}$.

\item A first, more accessible step beyond the sub-critical regime $\hat{\beta} < \hat{\beta}_c$ is to study the \emph{quasi-critical regime}
(introduced by \cite{CCR25} in the i.i.d.\@ case), which interpolates the sub-critical regime and the critical regime $\hat{\beta} = \hat{\beta}_c$ by taking $\hat{\beta} \uparrow \hat{\beta}_c $ slower than the critical window $\hat{\beta}^2 = \hat{\beta}_c^2 + O\big(\frac{1}{\log N}\big)$.

\item For $\hat \beta < \hat \beta_c$, higher moments $\IE[W_N(\beta_N)^q]$, with a parameter $q=q_N$ that may diverge with $N$ are of interest, see e.g.\ \cite{CN25,cosco,CZ23} in the case of space-time independent noise.
We note that this question is the first step regarding the study of extremal value statistics of the partition function through their connection to log-correlated fields, see \cite{CNZ25,CZ23} for more details. 

\item Moreover, for $\hat \beta < \hat \beta_c$ we conjecture that $W_N(\beta_N)$, should be bounded in every $L^q$, $q\in \mathbb N$, see \cite{LyZy21,LyZy22} in the independent noise case. The second reference suggests that a multivariate extension of Theorem \ref{th:ErdosTaylor} should hold. 
\end{itemize}

\section*{Appendix}\label{Appendix}
\appendix
\setcounter{section}{0}
\setcounter{equation}{0}

\section{Proof of Lemmas \ref{lem:order_of_RNx} and \ref{lem:bound_h}} \label{app:techtools}
In the following, we provide the proofs of the technical results stated in Section \ref{sec:technicaltools}.

\begin{proof}[Proof of Lemma \ref{lem:order_of_RNx}]
Take $x,y$ as in the first part of the statement of the lemma. We have
\[
    |x-y|\leq |x|+|y|
    \leq 2N^{(l'+1)/2M}\,,
\]
and
\[
    |x-y|\geq \big||x|-|y|\big|\geq N^{l'/2M}-N^{(l+1)/2M}\geq \frac{1}{2}N^{l'/2M}\,.
\]
By the Local Limit Theorem (see \cite{book}), denoting $\bar p_n(x) \coloneqq (\pi n)^{-1} e^{-|x|^2/2n}$ we obtain
    \begin{align*}
    R_{N^{\nu},N^{\xi}}(x-y) &=\sum_{n=N^{\nu}+1}^{N^\xi}p_{2n}(x-y)=
    \sum_{n=N^{\nu}+1}^{N^\xi} \bar p_n(x-y)(1+o_N(1))
    \\
    &=\sum_{n=N^{\nu}+1}^{N^\xi} \bar p_n(x-y)
    +o_N(\log N)\,, \label{eq:seclineRN}
    \end{align*}
    where $o_N(1)$ as $N \to \infty$ depends on $\nu, l, l'$. (We drop the integer part notation for simplicity.)
    
    We distinguish
    two cases.
 If $\nu<l'/M<\xi$, we can bound the first term in \eqref{eq:seclineRN} as
    \begin{align}\nonumber
        \sum_{n=N^{\nu}+1}^{N^\xi}\bar p_n(x-y)
        &\leq \sum_{n=N^\nu+1}^{N^{(l'-1)/M}} \frac{1}{\pi n} {e^{-\frac{N^{l'/M}}{8N^{(l'-1)/M}}}}  + \sum_{n=N^{(l'-1)/M}}^{N^{\xi}} \frac{1}{\pi n}\\
        &\leq \sum_{n=N^{(l'-1)/M}}^{N^{\xi}} \frac{1}{\pi n}  + o_N(\log N),
        \nonumber \\
        & \le \left( \xi - \frac{l'}{M} + \frac{1}{M} \right) \frac{\log N}{\pi} + o_N(\log N)\,.
    \end{align}
    where $o_N(\log N)$ depends on $M, \xi, \nu$. Similarly, for for every $\delta>0$ and $N$ large enough,
    \begin{align} \nonumber
         \sum_{n=N^{\nu}+1}^{N^\xi}\bar p_n(x-y)&
         \geq \sum_{n=N^{(l'+1)/M}}^{N^{\xi}} \frac{1}{\pi n} {e^{-\frac{2N^{l'/M}}{N^{(l'+1)/M}}}}
        \geq \sum_{n=N^{(l'+1)/M}}^{N^{\xi}} \frac{(1-\delta)}{\pi n} \\ &\label{eq:upper_bound_2}
        \geq \left( \xi - \frac{l'}{M} -\frac{1}{M} \right)\frac{\log N}{\pi} (1-\delta) + o_N(\log N)\,.
    \end{align}
    Now, if $l'/M<\nu$, by analogy with the previous computations we obtain    \begin{equation}\label{eq:lower_bounds}
    \begin{split}
    &R_{N^{\nu},N^{\xi}}(x-y)\leq(\xi-\nu+\frac{2}{M})\frac{\log N}{\pi}
    \,, \\
        &R_{N^{\nu},N^{\xi}}(x-y) \geq (\xi-\nu- \frac{3}{M})\frac{\log N}{\pi}\,,
    \end{split}
    \end{equation}
    where \eqref{eq:lower_bounds} follows once noticed that $l'/M<\nu$ implies $|x-y|\leq 2N^{\nu/2 + 1/2M}$.
    
    The bounds in \eqref{eq:order_RN_more_general} can be derived via similar techniques by observing that $R_{N^\xi}(x)$ is asymptotically decreasing in $|x|$.

    We turn to \eqref{eq:ChernoffBound}.
    For $N_0(M)$ large enough, we have for all $N\geq N_0$,
\begin{equation*}
    |x - y| \ge N^{\xi + \frac{1}{2M}} - N^{\xi}
    \ge \frac{1}{2} N^{\xi + \frac{1}{2M}}\,,
\end{equation*}
so that
the left-hand side of \eqref{eq:ChernoffBound} is less than
\begin{align}
    &  \, \sum_{n=1}^{N^\xi} P \big( \,|S_{2n}| > N^{\xi + \frac{1}{2M}}\,\big) \leq N^{\xi} e^{-cN^{1/M}}\,,
\end{align}
where $c>0$ and the inequality follows from Chernoff's inequality.
\end{proof}

\medskip

\begin{proof}[Proof of Lemma \ref{lem:bound_h}]
Let $0\leq u<v\leq 1$. We first claim that as $N\to\infty$, one has
\begin{equation}\label{eq:firstsumh}
    \sum_{N^u<|x| \leq N^v,\, x\in \mathbb Z^2_{even}} \frac{\log (|x|)^a}{|x|^2} \sim \frac{\pi}{a+1} (\log N)^{a+1}(v^{a+1}-u^{a+1})\,. 
\end{equation}
Indeed, observe that $\sum_{0<|x| \leq N,\,x\in \mathbb Z^2_{even}} \frac{\log (|x|)^a}{|x|^2}$ and $\sum_{0<|x| \leq N,\,x\in \mathbb Z^2_{odd}} \frac{\log (|x|)^a}{|x|^2}$ only differ by an order one term. (This can be checked using that $\sum_{x\in \mathbb Z^2\setminus\{0\}} \frac{\log (|x|)^a}{|x|^3}$ converges.) Besides, a comparison with an integral yields that $\sum_{0<|x| \leq N,x\in \mathbb Z^2} \frac{\log (|x|)^a}{|x|^2}$ is equivalent to
$\int_{[1,N]^2} \frac{\log (|x|)^a}{|x|^2}\mathrm{d} x = \frac{2\pi}{a+1}(\log N)^{a+1}$.

Let $\varepsilon >0$. By \eqref{eq:firstsumh}, Taylor's inequality and the mean value theorem, we obtain for $N$ large enough
\begin{equation*}
(1-\varepsilon)\pi (\log N)^{a+1} (u^a \wedge v^a) (v-u) \leq \sum_{\substack{N^u<|x| \leq N^v \\ x\in \mathbb Z^2_{even}}} \frac{\log (|x|)^a}{|x|^2} \leq (1+\varepsilon) \pi (\log N)^{a+1} (u^a \vee v^a) (v-u).
\end{equation*}
Thus, there exists  $N_0'=N_0(M,\varepsilon)$ such that for all $N\geq N_0'$ and $\forall \,l\in\{1,\dots,M\}$, 
\begin{equation} \label{eq:boundlogx2}
(1-\varepsilon)\pi\frac{\left(l+\mathds{1}_{a\leq0}\right)^a}{(2M)^a}   \frac{(\log N)^{a+1}}{2M}\leq  \sum_{x\in A_{l}} \frac{\log (|x|)^a}{|x|^2} \leq
(1+\varepsilon) \pi\frac{\left(l+\mathds{1}_{a>0}\right)^a}{(2M)^a} \frac{(\log N)^{a+1}}{2M} \,.
\end{equation}
Moreover, by definition of $h$, we can find $R=R(\varepsilon)>0$ such that
\begin{equation} \label{eq:boundOnh}
 \forall |x|>R,\quad (1-\varepsilon)  \leq  h(x)\, \frac{|x|^2}{ \log (|x|)^a}  \leq  (1+\varepsilon)\,.
\end{equation}
Hence \eqref{eq:sumhAl} follows from \eqref{eq:boundOnh} and \eqref{eq:boundlogx2}. For $l=0$, we split the sum and use \eqref{eq:firstsumh}:
\[
\sum_{x\in A_0}h(x) = \sum_{|x|\leq R} h(x) + \sum_{A_0\setminus B(R)} h(x) \leq \Vert h \Vert_{\infty} R^2 + \frac{\pi}{a+1}\left(\frac{\log N}{2M}\right)^{a+1}.
\]
\end{proof}

\section{$L^2$-approximation in product form}\label{sec:approximation}
Define $Y_N \coloneqq E[\prod_{n=1}^N \{1+\beta_N \omega(n,S_n)\}]$. 
\begin{lemma} \label{lem:Ups}
Let $\hat \beta < z_a$. We have $ Y_N - W_N \xlongrightarrow[N \to \infty]{L^2}  0$.
\end{lemma}
\begin{proof}
For $\omega,\omega'$ two centered, unit variance, jointly Gaussian random variables of covariance $h$, we have $\IE[e^{\beta \omega-\frac{\beta^2}{2}}\omega'] = e^{-\frac{\beta^2}{2}}\frac{\mathrm d}{\mathrm d\beta'} \IE[e^{\beta \omega+\beta'\omega'}]_{|\beta'=0} = \beta h$. Hence,
\[
\IE[W_N Y_N] = E_0^{\otimes 2}\left[\prod_{n=1}^N\left\{1+ \beta_N \mathbb{E}[\eta(n,S_n^1)\omega(n,S_n^2)]\right\}\right]=E\left[\prod_{n=1}^N\left\{1+ \beta_N^2 h(S_{2n})]\right\}\right],
\]
and
\[
\IE\left[(W_N-Y_N)^2\right] = E\left[e^{\beta_N^2 \sum_{n=1}^N h(S_{2n})} \right] - E\left[\prod_{n=1}^N\left\{1+\beta_N^2 h(S_{2n})\right\}\right].
\]
By
Theorem \ref{thm:main} and Lemma \ref{prop:lowerBound}, the right hand side (which is positive) goes to zero.
\end{proof}
\bibliographystyle{plain}
{\footnotesize \bibliography{polymeres-bib}
}

\end{document}